\documentclass{amsart}
\usepackage{stmaryrd}
\usepackage{amssymb}
\usepackage{hyperref}
\usepackage{txfonts}
\usepackage{amsfonts}
\usepackage{hyperref}
\usepackage{amsfonts}
\usepackage{mathrsfs}

\newtheorem{theorem}{Theorem}[section]
\newtheorem{lemma}[theorem]{Lemma}
\newtheorem{proposition}[theorem]{Proposition}
\newtheorem{corollary}[theorem]{Corollary}

\theoremstyle{definition}
\newtheorem{definition}[theorem]{Definition}
\newtheorem{example}[theorem]{Example}

\theoremstyle{remark}
\newtheorem{remark}[theorem]{Remark}

\numberwithin{equation}{section}

\begin{document}

\title[Geometric flows on warped product manifold]
{Geometric flows on warped product manifold }

\author{Wei-Jun Lu}
\address{Center of Mathematical Sciences, Zhejiang University,
 Hangzhou, 310027, China; School of Science,
 Guangxi University for Nationalities, Naning, 530007, China }
 \email{weijunlu2008@126.com}



\subjclass[2000]{53C21, 53C25, 53C44, 58J45 }

\date{}


\keywords{warped product metric; Ricci flow; hyperbolic geometric
flow; warped product solution; evolution equation}

\begin{abstract}
  {\footnotesize  We derive one unified formula for Ricci
curvature tensor on arbitrary warped product manifold by introducing
a new notation for the lift vector and the Levi-Civita connection.
This formula is helpful to further consider Ricci flow (RF) and
hyperbolic geometric flow (HGF) and evolution equations on warped
product manifold. We characterize the behavior of warping function
under RF and under HGF. Simultaneously, we give some simple examples
to illustrate the existence of such warping function solution. In
addition, we also gain the evolution equations for metrics and Ricci
curvature on a general warped product manifold and specific warped
product manifold whose second factor manifold is of Einstein metric.
}
\end{abstract}

\maketitle

\section{Introduction}

   From Riemann's work it appears that he worked with changing metrics mostly
by multiplying them by a function (conformal change). Soon after
Riemmann's discoveries it was realized that in polar coordinates one
can change the metric in a different way, now referred to as a
warped product metric (WPM). The concept of warped product metrics
was first introduced by Bishop and O'Neill \cite{BO'N} to construct
examples of Riemannian manifolds with negative curvature. In
Riemannian geometry, warped product manifolds and their generic
forms have been used to construct new examples with interesting
curvature properties like Einstein spaces \cite{Be, KK} or (locally)
symmetric spaces \cite{BG}. In string theory, Yau in
\cite[P244-245]{YN} argued that ``the easiest way to partition the
ten-dimensional space is to cut it cleanly, splitting it into
four-dimensional spacetime and six-dimensional hidden subspace, and
in the non-k\"{a}hler case, the ten-dimensional spacetime is not a
Cartesian product but rather a warped product."

In this paper, we shall consider the warped product metrics
combining with two types of geometric flows, i.e. Ricci flow (RF)
and hyperbolic geometric flow (HGF).

 As we have known, Ricci flow was
introduced and studied by Hamilton \cite{Ha}. This was the first
means to study the geometric quantities associated to a metric $g(x,
t),\,(x, t) \in  M \times \mathbb{R}$ as the metric evolves via a
PDE, where $M$ is a differentiable manifold. The Ricci flow is a
powerful tool to understand the geometry and topology of some
Riemann manifolds. Any solution of Ricci flow equation will help us
to understand its behavior for general cases and the singularity
formation, further the basic topological and geometrical properties
as well as analytic properties of the underlying manifolds. On the
other hand, a hyperbolic Ricci evolution is the Ricci wave, i.e.
hyperbolic geometric flow (HGF) introduced by Kong and Liu
\cite{KL}. In fact, both RF and HGF can be viewed as prolongations
of the Einstein equation, whose left-hand side consists of what's
called the modified Ricci tensor. Since the right-hand side of the
RF and HGF equation also includes a key term in the famous Einstein
equation---the Ricci curvature tensor which shows how matter and
energy affect the geometry of spacetime, HGF, RF and Einstein
equation can be unified into a single PDEs system as
\begin{equation} \label{10-11}
 \alpha(x,t)\frac{\partial^2}{\partial t^2}g(t)
  +\beta(x,t)\frac{\partial}{\partial t}g(t)+\gamma(x,t)g(t)+2Ric_{g(t)}=0,
  \end{equation}
where $\alpha(x,t),\,\beta(x,t),\,\gamma(x,t)$ are certain smooth
functions (\cite{KL}, \cite{HU}). It is easy to see from
(\ref{10-11}) that the above three cases correspond to
$``\alpha(x,t)=1,\,\beta(x,t)=\gamma(x,t)=0"$,
$``\alpha(x,t)=0,\,\beta(x,t)=1,\,\gamma(x,t)=0"$ and
$``\alpha(x,t)=0,\,\beta(x,t)=0,\,\gamma(x,t)=const"$, respectively.

In the topic of combining geometric flow with warped product
manifolds, there have been made some progress recently. For
instance, Ma and Xu in \cite{MX} showed that the negative curvature
is preserved in the deformation of hyperbolic warped product metrics
under RF by such example: $\overline M = \mathbb R _+ \times N^n$
with the product metric
 $ g(t) = \varphi(x, t)^2 dx^2 + \psi(x, t)^ 2 \hat g $, where
 $ (N^n, \hat g)$ is an Einstein manifold of dimension $ n \geq 2 $,
 $\varphi(x)$ and $\psi(x)$ are two smooth positive functions of the variable
 $ x > 0$. Xu and Ma's work is mainly inspired from the work of Simon
 \cite{Si}.  Das, Prabhu and Kar in their work \cite{DPK} mainly
 considered the evolution under RF of the
warped product $\mathbb R^1\times M $ with line element of the form
 $$ds^2 = e^{2f}(\sigma, \lambda) (-dt^2 + dx^2 + dy^2 + dz^2 )
         + r^2_ c (\sigma, \lambda)d\sigma^2 $$
and the behavior of $f$ by solving the flow equations, where $M$ is
Minkowski spacetime and $\mathbb R^1$ is the real line, $\lambda$ is
flow parameter. Especially, Simon \cite{Si} characterized the local
existence of Ricci flow on the complete non-compact manifold
$X=(\mathbb{R},h)\times (N^n, \gamma)$ with warped product metric
$g(x,q)=h(x)\oplus r^2(x)\gamma(q)$ and showed that {\em if
$g_0(x,q)=h_0(x)\oplus r^2_0(x)\gamma(q)$ is arbitrary warped
product metric which satisfies some certain conditions
  \begin{equation*}\left\{\begin{array}{ll}
 & \sup\limits_{x\in \mathbb{R}} (h_0)_{xx} <\infty,
     \quad \inf\limits_{x\in \mathbb{R}} (h_0)_{xx}>0,\quad \inf\limits_{x\in \mathbb{R}} r_0(x) >0,\\
 & \sup\limits_{x\in \mathbb{R}}\big( |(\frac{\partial}{\partial x})^j h_0(x)|
          +|(\frac{\partial}{\partial x})^j \log r_0(x)| \big) <\infty, \forall j\in \{1,2,\ldots\},
   \end{array} \right.
   \end{equation*}
 then there exists a unique warped product solution $g(x,q,t)=h(x,t)\oplus
 r^2(x,t)\gamma(q)$, $t\in [0,T)$ to the Ricci flow
      $$
      \frac{\partial }{\partial t}g(t)=-Ric_{g(t)},\quad g(0)=g_0,\quad t\in [0,T).
      $$}
For more detail we refer to see Theorem 3.1 in \cite{Si}.

 Motivated by \cite{MX}, \cite{Si} and
\cite{DPK}, we are interested in the behavior of geometric flows
associated to the general WPM
$\overline{M}=(M_1,g_1)\times_{\lambda} (M_2,g_2)$. At the same
time, we have paid attention to a known fact that ``{\em if
$(M_1,g_1(t))$ and $(M_2, g_2(t)$ are solutions of the Ricci flow on
a common time, then their direct product $(M_1\times M_2,\ g_1(t)+
g_2(t))$ is a solution to the Ricci flow}" (see Exercise 2.5 in
\cite{CLN}, P99). Naturally, we wish to generalize this result to
warped product manifold and even to hyperbolic geometric flow as
well. Realizing that the Ricci curvature tensor formula on WPM is of
vital role in studying the Ricci flow and hyperbolic geometric flow,
we first integrate the separated Ricci curvature formula in previous
academic literature. Since the formulas about Riemann curvature and
Ricci curvature are divided up into several parts according to the
horizontal lift or vertical lift of the tangent vectors attached to
$M_1$ or $M_2$ (see Propositions 2.7 and 2.9). To better study the
RF and HGF associated to WPM,  regardless of the tangent vectors are
attached to horizontal lift or vertical lift, we have to derive out
one formula as a whole. By introducing a new notation for lift
vector (see Proposition \ref{9-23-1}, Remark \ref{9-23-2}) and
Levi-Civita connection $\bar \nabla$ over $\overline M$, we derive a
unified formula (\ref{6-30-6}) for Ricci curvature and scalar
curvature (see Theorem \ref{6-30-4}). Using this unified Ricci
curvature formula, we consider the behavior of warping function
under the RF and under HGF on warped product manifold $\overline M $
(unnecessarily compact)and give two main results (Theorem
\ref{5-15-0}, Theorem \ref{5-24-2}), which assert that the warping
function $\lambda$ should satisfy a characteristic equation when the
warped product metric $\bar g(x,y,t)=g_1(x,t) \oplus \lambda^2(x,t)
g_2(y,t)$ is also a solution to the RF (resp. HGF), where $g_1(x,t)$
and $g_2(y,t)$ are respectively solutions to the RF (resp. HGF) .

Considering one may worry about these equations have no any solution
$\lambda$, we make some appropriate illustration. We employ the
following two strategies for overcoming this obstruction: one is by
appealing to the known short-time existence theorem of geometric
flows (in compact case, refer to \cite{Ha}, \cite{De}, \cite{DKL};
in complete no-compact case, refer to \cite{Sh} ); another is to get
some sense by constructing some specific examples (see Example
\ref{9-30-2} and example \ref{7-2-07}), whose ideas mainly come from
\cite{Pe,Si,AK,MX}.

 In addition, in order to understand how the curvature on warped product manifold
is evolving and behaving , using the unified Ricci curvature formula
(\ref{6-30-6}) we also consider the evolution equations along the RF
and HGF. On general WPM, we derive two class results: (1) the
evolution equations for metric and warping function, see Proposition
\ref{10-5-01} and Proposition \ref{10-5-05};
 and (2) the Ricci curvature evolution equations (\ref{10-5-08}),(\ref{10-5-09}) in
Theorem \ref{10-5-07}. On a specific warped product manifold whose
warped product metric is of the form $\bar g(x,y,t)=g_1(x,t) \oplus
\lambda^2(x,t)g_2(y,t)\big)$ with a fixed Einstein
 metric $g_2$, we gain the more interesting evolution equations
 for Ricci curvature and special function $f(x,t)$, see (\ref{10-5-09'}) and
 (\ref{10-6-14}) in Theorem \ref{10-5-009} and Theorem \ref{10-6-13'}.

 The organization of this paper is below. In Section 2, we review some preliminary results and
 derive out three unified formulas for Riemannian curvature, Ricci
 curvature and scalar curvature on warped product manifold. Section 3 is devoted to
characterize the behavior of warping function under the Ricci flow.
We give a necessary and sufficient condition for the warping
function $\lambda$, including the elaboration on the short-time
existence of warped product solution to the RF. Furthermore, we also
present some examples. Section 4 is parallel to Section 3. A
distinction between them is in that the considered flow is HGF
rather than RF. In last Section, we discuss evolution equations of
warping function and Ricci curvature on general and specific WPM
under the RF and under the HGF.

\vskip 3mm{\bf Acknowledgments.} The author thanks Professor Kefeng
Liu for constructive suggestion. He would like to thank Doctor
Jian-Ming Wan, Doctor Hai-Ping Fu, Professor De-Xing Kong, Doctor
Chang-Yong Yin and Professor Fangyang Zheng for some useful
discussion.

\section{Unified Ricci tensor formula on WPM}

Before studying the RF and HGF of warped product metrics, we need to
deduce the crucial formula for Ricci tensor from the separated form
to united form on warped product manifolds. We do this by the
construct of the connection. As we will see, this unified Ricci
tensor formula simplifies the study of curvature tensors associated
warped product metrics, and also allows us to find explicit formulas
for RF and HGF with respect to a given underlying warped product
manifolds.

We first introduce background knowledge on warped product manifolds,
see \cite{BO'N, O'N} for detail.

\subsection{Basics of warped products}

 Let $M_1$ and $M_2$ be Riemannian manifolds equipped with Riemannian
metrics $g_1$ and $g_2$, respectively, and let $\lambda$ be a
strictly positive real function on $M_1$. Consider the product
manifold $M_1 \times M_2$ with its natural projections $\pi_1 : M_1
\times M_2 \to M_1$ and $\pi_2 : M_1 \times M_2 \to M_2$.

The warped product manifold $\overline{M} = M_1 \times_{\lambda}
M_2$ is the manifold $M_1 \times M_2$ equipped with the Riemannian
metric $\bar g=g_1 \oplus \lambda^2 g_2$ defined by
  $$\bar{g} (X,Y) = g_1(d\pi_1(X), d\pi_1(Y )) + \lambda^2 g_2 (d\pi_2(X), d\pi_2(Y ))$$
 for any tangent vectors $X, Y \in T(x,y)(M_1\times M_2)$. The function $\lambda$ is called
 the warping function of the warped product. When $\lambda=1$,  $M_1 \times_{\lambda} M_2$ is a direct
product.

For a warped product manifold $M_1 \times_{\lambda} M_2$, $M_1$ is
called the base and $M_2$ the fiber. The fibers $p\times
M_2=\pi_1^{-1}(p)$ and the leaves $M_1\times
 q=\pi_2^{-1}(q)$ are Riemannian submanifolds of $\overline{M}$. Vectors
 tangent to leaves are called horizontal and those tangent to fibers
 are called vertical. We denote by $\mathscr{H}$ the orthogonal
 projection of $T_{(p,q)}\overline{M}$ onto its horizontal subspace $T_{(p,q)} M_1\times
 q$, and by $\mathscr{V}$ the projection onto the vertical subspace $T_{(p,q)} p
 \times M_2$.

If $v\in T_pM_1$, $p\in M_1$ and $q\in M_2$, then the lift $\tilde
v$ of $v$ to $(p, q)$ is the unique vector in $T_{(p,q)}
M_1=T_{(p,q)}M_1\times q \subset T_{(p,q)}\overline{M}$ such that
$d\pi_1(\tilde v) =v$. For a vector field $X\in \mathscr{X}(M_1)$,
the lift of $X$ to $\overline{M}$ is the vector field $\tilde{X}$
whose value at each $(p,q)$ is the lift of $X_p$ to $(p,q)$. The set
of all such horizontal lifts is denoted by $\mathscr{L}(M_1)$.
Similarly, we denote by $\mathscr{L}(M_2)$ the set of all vertical
lifts.

We state some known results below.

\begin{proposition} \label{5-22-2}
(1) If $\tilde X, \tilde Y \in \mathscr L(M_1)$ then
$$[\tilde X,\tilde Y]=[X,Y]^{\sim} \in \mathscr L(M_1);$$
(2) If $\tilde U,
\tilde V \in \mathscr{L}(M_2)$ then
$$[\tilde U,\tilde V]=[U,V]^{\sim} \in \mathscr{L}(M_2);$$
(3) If $\tilde X \in \mathscr{L}(M_1)$ and $ \tilde V \in
\mathscr{L}(M_2)$ then $[\tilde X, \tilde V]=0$.

\end{proposition}

\begin{proposition}(\cite{O'N}, Prop.35, P206) \label{5-23-1}
On $\overline{M}$, if $X,Y \in \mathscr{L}(M_1)$ and
$V,W\in \mathscr{L}(M_2)$, then \\
(1) $\bar \nabla_X Y \in \mathscr{L}(M_1)$ is the lift of
$\;^{M_1}\!\nabla_X Y $ on $M_1$;\\
 (2) $\bar \nabla_X V=\bar \nabla_V X= \frac{X\lambda}{\lambda} V$.\\
(3) $ \mathrm{nor}\bar\nabla _V W=II(V,W)=-\frac{<V,W>}{\lambda}
\mathrm{grad} \lambda$, where
 $$nor: \mathscr{H}\to T_{(p,q)}(M_1\times q)=\big(T_{(p,q)} p\times M_2\big)^\perp. $$\\
\noindent(4) $\mathrm{tan} \bar \nabla_V W\in \mathscr{L}(M_2)$ is
the lift of $\;^{M_2}\!\nabla_V W$ on $M_2$, where
  $$tan: \mathscr{V} \to T_{(p,q)} (p\times M_2). $$
\end{proposition}

Let $\;^{M_1}\!R$ and $\;^{M_2}\!R$ be the lifts on $\overline{M}$
of the Riemannian curvature tensors of $M_1$ and $M_2$,
respectively. Since the projection $\pi_1$ is an isometry on each
leaf, $\;^{M_1}\!R$ gives the Riemannian curvature of each leaf. The
corresponding assertion holds for $\;^{M_2}\!R$, since the
projection $\pi_2$ is a homothety. Because leaves are totally
geodesic, $\;^{M_1}\!R$ agrees with the curvature tensor $\bar R$ of
$\overline{M}$ on horizontal vectors. This time the corresponding
assertion fails for $\;^{M_2}\!R$ and $\bar R$, since fibers are in
general only umbilic. In addition, for convenience the alternative
notation $\bar R(X,Y)Z$ is $\bar R_{XY}Z$.

\begin{proposition}(\cite{O'N}, Prop.42, P210) \label{5-23-2}
Let $\overline{M}$ be a warped product manifold, if $X,Y, Z \in
\mathscr{L}(M_1)$ and
$U, V, W\in \mathscr{L}(M_2)$, then \\
(1) $\bar R_{X Y}Z \in \mathscr{L}(M_1)$ is the lift of
$\;^{M_1}\!R_{XY}Z $ on $M_1$;\\
 (2) $\bar R_{VX}Y=\big(\mathrm{Hess}(\lambda)(X, Y)/\lambda \big) V$.\\
(3) $ \bar R_{XY}V=\bar R_{VW}X=0.$\\
(4) $ \bar R_{XV}W =\frac{\bar g(V,W)}{\lambda} \bar\nabla_X
\mathrm{grad}\lambda.$ \\
(5) $\bar R_{VW}U =\;^{M_2}\!R_{V W}U- \frac{1}{\lambda^2}\bar
g(\mathrm{grad}\lambda, \mathrm{grad}\lambda)\big(\bar g(V,U) W-\bar
g(W,U)V\big)$.
\end{proposition}

Writing $\;^{M_1}\!\mathrm{Ric}$ for the lift (pullback by $\pi_1:
\overline{M} \to M_1$) of the Ricci curvature of $M_1$, and
similarly for $\;^{M_2}\!\mathrm{Ric}$.

\begin{proposition}(\cite{O'N},Corollary 43, P211) \label{5-23-3}
On a warped product $\overline{M}$ with $m_2=\mathrm{dim}
M_2 >1$, let $X,Y$ be horizontal and $V, W$ vertical. Then \\
  (1)  $ \overline{\mathrm{Ric}}(X,Y)=\;^{M_1}\!\mathrm{Ric}(X,Y)
    -\frac{m_2}{\lambda}\mathrm{Hess}(\lambda)(X, Y)$.\\
 (2) $\overline{\mathrm{Ric}}(X,V)=0.$\\
 (3) $ \overline { \mathrm{Ric}}(V,W)=\;^{M_2}\!\mathrm{Ric}(V,W)-\bar g(V,W)\lambda^{\#}$,
 where
$$\lambda^{\#}=\frac{\Delta \lambda}{\lambda}
   +(m_2-1)\frac{\bar g(\mathrm{grad}\lambda, \mathrm{grad}\lambda)}{\lambda^2}$$
and $\Delta \lambda=\mathrm{Tr}(Hess(\lambda))$ is the Laplacian on
$M_1$.
\end{proposition}

\subsection{The unified formulas for Ricci curvature}
 From the precious subsection we have seen that the formulas about
Riemann curvature and Ricci curvature are divided up into several
parts according to the horizontal lift or vertical lift of the
tangent vectors attached to $M_1$ or $M_2$. To better study the RF
and HGF associated to WPM, we feel it is necessary to derive one
unified formula for Ricci tensor, no matter how the lift vectors are
either horizontal or vertical. For this we first introduc the
unified connection and unified Riemannian curvature on a general
warped product manifold $\overline{M}$ ( cf. \cite{BG}, \cite{BMO})
by introducing a new notation of lift vector.

\begin{proposition} \label{9-23-1} Let $X=(X_1,X_2),\, Y=(Y_1,Y_2)\in \mathscr{X}(\overline{M})$,
where $X_1,X_2 \in \mathscr{X}(M_1)$ and
 $Y_1, Y_2 \in \mathscr{X}(M_2)$. Denote $\nabla$ by the Levi-Civita connection on the Riemannian
 product $M_1\times M_2$ with respect to the direct product metric $g=g_1\oplus g_2$ and by
 $R$ its curvature tensor field. Then the Levi-Civita connection $\bar{\nabla}$ of $\overline{M}$ is given by
  \begin{equation} \label{5-6-1}
  \begin{array}{ll}
  \bar{\nabla}_X Y &= \nabla_X  Y
                  + \frac{1}{2\lambda^2}X_1(\lambda^2)(0,Y_2)\\
   &               + \frac{1}{2\lambda^2}Y_1(\lambda^2)(0,X_2) -\frac{1}{2}g_2(X_2,Y_2)(grad\,\lambda^2,0)\\
  &=\big(\,^{M_1}\!\nabla_{X_1}Y_1-\frac{1}{2}g_2(X_2,Y_2)grad\,\lambda^2,0 \big)\\
   &+\big(0, \,^{M_2}\!\nabla_{X_2}Y_2+ \frac{1}{2\lambda^2}X_1(\lambda^2)Y_2
   +\frac{1}{2\lambda^2}Y_1(\lambda^2)X_2 \big),\\
    \end{array}
   \end{equation}
and the relation between the curvature tensor fields of
$\overline{M}$ and $M_1 \times M_2$ is
  \begin{equation} \label{5-6-2}
 \begin{split}
    \bar R_{XY}-R_{XY} & = \frac{1}{2\lambda^2} \Big\{
            \Big(\,^{M_1}\!\nabla_{Y_1}grad_{g_1}\lambda^2
          - \frac{1}{2\lambda^2}Y_1(\lambda^2)grad_{g_1}\lambda^2,\,0 \Big)
           \wedge_{\bar g} (0, X_2)  \\
      & - \Big(\,^{M_1}\!\nabla_{X_1}grad_{g_1}\lambda^2
          - \frac{1}{2\lambda^2}X_1(\lambda^2)grad_{g_1}\lambda^2,\,0\Big)
           \wedge_{\bar g} (0, Y_2)\\
      & -\frac{1}{2\lambda^2}|grad_{g_1}\lambda^2|^2 (0,X_2)\wedge_{\bar g}(0,
      Y_2)\Big\}
 \end{split}
 \end{equation}
where the wedge product $(X\wedge_{\bar g} Y)Z =\bar g(Y,Z)X-\bar
g(X,Z)Y$, for all  $X, Y, Z \in \mathscr{X}(\overline{M})$.

\end{proposition}

\begin{remark} \label{9-23-2} We can easily show that the four cases in Proposition \ref{5-23-1} can be
integrated to one form as (\ref{5-6-1}), where we denote the lifts
of $X_1\in \mathscr{X}(M_1), X_2\in \mathscr{X}(M_2)$ by $(X_1,0),
(0,X_2)\in \mathscr{X}(\overline{M})$. For example,
 \begin{equation*}
 \begin{split}
  &\bar\nabla_{(X_1,0)} (Y_1,0)=(\,^{M_2}\!\nabla_{X_1} Y_1,0)=\mathrm{lift\ of\ }\,^{M_1}\!\nabla_{X_1} Y_1,\\
  & \bar\nabla_{(X_1,0)} (0,V_2)=\bar\nabla_{(0,V_2)} (X_1,0)=\frac{X_1(\lambda)}{\lambda}(0,V_2),\\
  & \bar\nabla_{(0,V_2)} (0,W_2)=(-\frac{1}{2}g_2(V_2,W_2)\mathrm{grad\ }\lambda^2,0)+(0, \,^{M_2}\!\nabla_{V_2} W_2),\\
   &\mathrm{nor}\bar\nabla_{(0,V_2)}
   (0,W_2)=-\frac{1}{2}g_2(V_2,W_2) (\mathrm{grad}\lambda^2,0)\\
  &  \hskip 2cm  =-\frac{\bar g((0,V_2),(0,W_2))}{\lambda} (\mathrm{grad}\lambda,0),\\
   &  \mathrm{tan }\bar\nabla_{(0,V_2)} (0,W_2)=(0,\,^{M_2}\!\nabla_{V_2} W_2)
       =\mathrm{lift\ of\ }\,^{M_2}\!\nabla_{V_2} W_2.
\end{split}
 \end{equation*}
\end{remark}

From (\ref{5-6-2}), we easily obtain
 \begin{proposition}\label{6-30-1}
 \begin{equation} \label{5-6-3}
 \begin{split}
&\bar{R} _{(X_1, X_2)(Y_1,Y_2)} (Z_1,Z_2)\\
    &  =(\,^{M_1}\!R_{ X_1 Y_1} Z_1,\,^{M_2}\!R_{ X_2 Y_2} Z_2)\\
    & +\frac{1}{2} g_2(X_2,Z_2) \big( \,^{M_1}\!\nabla_{Y_1} \mathrm{grad}\ \lambda^2
       -\frac{1}{2\lambda^2}  Y_1(\lambda^2)\mathrm{grad}\ \lambda^2,0 \big)\\
  & -\frac{1}{2} g_2(Y_2,Z_2) \big( \,^{M_1}\!\nabla_{X_1} \mathrm{grad}\ \lambda^2
             -\frac{1}{2\lambda^2} X_1(\lambda^2)\mathrm{grad}\ \lambda^2 ,0 \big)\\
  &  +\Big(0,  \frac{1}{2\lambda^2} g_1 ( \,^{M_1}\!\nabla_{X_1} \mathrm{grad}\ \lambda^2
            -\frac{1}{2\lambda^2}X_{1}(\lambda^2)\mathrm{grad}\ \lambda^2,\,Z_1 )Y_2 \big) \\
  &  -\Big(0,  \frac{1}{2\lambda^2} g_1 ( \,^{M_1}\!\nabla_{Y_1} \mathrm{grad}\ \lambda^2
             -\frac{1}{2\lambda^2}Y_1(\lambda^2)\mathrm{grad}\ \lambda^2,\,Z_1)X_2 \big) \\
 & +\big(0,  \frac{1}{4\lambda^2} \mid\mathrm{grad}\ \lambda^2 \mid^2 g_2(X_2,Z_2)Y_2\Big)\\
  &  -\big(0,\frac{1}{4\lambda^2} \mid\mathrm{grad}\ \lambda^2 \mid^2 g_2(Y_2,Z_2)X_2\Big).
 \end{split}
 \end{equation}
\end{proposition}

\begin{corollary}\label{6-30-2}
 \begin{equation} \label{6-30-3}
 \begin{split}
&\bar{R} _{(X_1, X_2)(Y_1,Y_2)} (Z_1,Z_2)\\
    &  =(\,^{M_1}\!R_{ X_1 Y_1} Z_1,\,^{M_2}\!R_{ X_2 Y_2} Z_2)\\
    & +\lambda g_2(X_2,Z_2) \big( \,^{M_1}\!\nabla_{Y_1} \mathrm{grad}\ \lambda, 0 \big)
   -\lambda g_2(Y_2,Z_2) \big( \,^{M_1}\!\nabla_{X_1} \mathrm{grad}\ \lambda, 0 \big)\\
  &  + \frac{1}{\lambda}\mathrm{Hess}(\lambda)(X_1,Z_1)(0,Y_2)
     -\frac{1}{\lambda}\mathrm{Hess}(\lambda)(Y_1,Z_1)(0,X_2)\\
   & +\mid\mathrm{grad}\ \lambda \mid^2 g_2(X_2,Z_2)(0,Y_2)
    - \mid\mathrm{grad}\ \lambda \mid^2 g_2(Y_2,Z_2) (0,X_2).
 \end{split}
 \end{equation}
\end{corollary}

\begin{proof}
  Note that
  \[\begin{split}
   & \,^{M_1}\!\nabla_{X_1} \mathrm{grad}\ \lambda^2
       -\frac{1}{2\lambda^2}  X_1(\lambda^2)\mathrm{grad}\ \lambda^2\\
  & =\,^{M_1}\!\nabla_{X_1} (2\lambda \mathrm{grad}\ \lambda)
       -\frac{1}{2\lambda^2} 2\lambda X_1(\lambda) 2\lambda\mathrm{grad}\ \lambda \\
  &=2\lambda\,^{M_1}\!\nabla_{X_1}  \mathrm{grad}\ \lambda
  \end{split}\]
and
 \[\begin{split}
   &\frac{1}{2\lambda^2} g_1 \left( \,^{M_1}\!\nabla_{X_1} \mathrm{grad}\ \lambda^2
            -\frac{1}{2\lambda^2}X_{1}(\lambda^2)\mathrm{grad}\ \lambda^2,\,Z_1  \right )\\
     &=\frac{1}{\lambda}g_1\left( \,^{M_1}\!\nabla_{X_1}  \mathrm{grad}\ \lambda,\,Z_1 \right)\\
     &= \frac{1}{\lambda}\mathrm{Hess}(\lambda)(X_1,Z_1).
  \end{split}\]
 Exchanging $X_1$ for $Y_1$, we obtain
 \[ \,^{M_1}\!\nabla_{Y_1} \mathrm{grad}\ \lambda^2
       -\frac{1}{2\lambda^2}  Y_1(\lambda^2)\mathrm{grad}\ \lambda^2
       =2\lambda\,^{M_1}\!\nabla_{Y_1}  \mathrm{grad}\ \lambda,\]
 \[ \frac{1}{2\lambda^2} g_1 \left( \,^{M_1}\!\nabla_{Y_1} \mathrm{grad}\ \lambda^2
            -\frac{1}{2\lambda^2}Y_{1}(\lambda^2)\mathrm{grad}\ \lambda^2,\,Z_1  \right )
 =\frac{1}{\lambda}\mathrm{Hess}(\lambda)(Y_1,Z_1). \]
Putting these facts together, (\ref{5-6-3}) can reduce to
\begin{equation} \notag
 \begin{split}
&\bar{R} _{(X_1, X_2)(Y_1,Y_2)} (Z_1,Z_2)\\
    &  =(\,^{M_1}\!R_{ X_1 Y_1} Z_1,\,^{M_2}\!R_{ X_2 Y_2} Z_2)\\
    & +\lambda g_2(X_2,Z_2) \big( \,^{M_1}\!\nabla_{Y_1} \mathrm{grad}\ \lambda, 0 \big)
   -\lambda g_2(Y_2,Z_2) \big( \,^{M_1}\!\nabla_{X_1} \mathrm{grad}\ \lambda, 0 \big)\\
  &  + \frac{1}{\lambda}\mathrm{Hess}(\lambda)(X_1,Z_1)(0,Y_2)
     -\frac{1}{\lambda}\mathrm{Hess}(\lambda)(Y_1,Z_1)(0,X_2)\\
   & +\mid\mathrm{grad}\ \lambda \mid^2 g_2(X_2,Z_2)(0,Y_2)
    - \mid\mathrm{grad}\ \lambda \mid^2 g_2(Y_2,Z_2) (0,X_2),
 \end{split}
 \end{equation}
as claimed (\ref{6-30-3}).
\end{proof}

Thus we have the unified formulas for (0,4)-type Riemannian
curvature tensor $\overline{Rm}$, Ricci curvature
$\overline{\mathrm{Ric} }$ and saclar curvature
$\overline{\mathrm{Scal} }$.

\begin{theorem}\label{6-30-4} On a warped product $\overline{M}$ with $m_2=\mathrm{dim} M_2\geq 2$. Let
$(X_1,X_2),(Y_1,Y_2),(Z_1,Z_2),(W_1,W_2)\in \mathscr{X}(\overline{M})$. Then\\
(i)\quad (0,4)-type Riemannian curvature tensor $\overline{Rm}$
satisfies
 \begin{equation} \label{6-30-5}
 \begin{split}
 &\overline{\mathrm{Rm} } \big((W_1,W_2),(Z_1,Z_2), (X_1, X_2), (Y_1,Y_2)\big)\\
  &  =\,^{M_1}\!\mathrm{Rm}(W_1,Z_1, X_1,Y_1)+\lambda^2\,^{M_2}\!\mathrm{Rm}(W_2,Z_2, X_2,Y_2)\\
   &\hskip 6mm     +\lambda \mathrm{Hess}(\lambda)(Y_1,W_1)g_2(X_2,Z_2)
      -\lambda \mathrm{Hess}(\lambda)(X_1,W_1) g_2(Y_2,Z_2) \\
   &+ \lambda \mathrm{Hess}(\lambda)(X_1,Z_1)g_2(W_2,Y_2)
     - \lambda \mathrm{Hess}(\lambda)(Y_1,Z_1)g_2(W_2,X_2)\\
   & +\lambda^2 \mid\mathrm{grad}\ \lambda \mid^2 g_2(X_2,Z_2)g_2(W_2,Y_2)
    - \lambda^2 \mid\mathrm{grad}\ \lambda \mid^2 g_2(Y_2,Z_2)g_2(W_2,X_2).
    \end{split}
     \end{equation}

\noindent (ii)\quad The Ricci curvature tensor $\overline{Ric}$
satisfies
     \begin{equation}\label{6-30-6}
     \begin{split}
     \overline{\mathrm{Ric} } \big( (X_1, X_2), (Y_1,Y_2) \big)
  &=\,^{M_1}\!\mathrm{Ric}(X_1,Y_1)+\,^{M_2}\!\mathrm{Ric}(X_2,Y_2)\\
  &-\lambda g_2(X_2,Y_2)\Delta_{M_1}\lambda
  -\frac{m_2}{\lambda}\mathrm{Hess}(\lambda)(X_1,Y_1)\\
  &-(m_2-1)\mid\mathrm{grad}\ \lambda \mid^2 g_2(X_2,Y_2).
    \end{split}
     \end{equation}
\noindent (iii) \quad The scalar curvature $\overline{Scal}$ is
     \begin{equation}\label{6-30-7}
     \begin{split}
    \overline{\mathrm{Scal} }& =\,^{M_1}\!\mathrm{Scal}+\frac{1}{\lambda^2}\,^{M_2}\!\mathrm{Scal}\\
       & -\frac{2m_2}{\lambda}\Delta_{M_1}\lambda
   -\frac{m_2(m_2-1)}{\lambda^2}\mid\mathrm{grad}\ \lambda \mid^2.
\end{split}
 \end{equation}
\end{theorem}

\begin{proof}
(i)  Note that
     \[ \overline{\mathrm{Rm} } \big((W_1,W_2),(Z_1,Z_2), (X_1, X_2), (Y_1,Y_2)\big)
        =\bar g\big((W_1,W_2),\bar{R} _{(X_1, X_2)(Y_1,Y_2)} (Z_1,Z_2)\big) \]
and $\bar g=g_1 \oplus \lambda^2 g_2$. By (\ref{6-30-3}) and the
 property of $\mathrm{Hess}(\lambda)$, we immediately obtain (\ref{6-30-5}).

As to the assertions (ii) and (iii), let $\{e_j\}_{j=1}^{m_1} $ be a
local orthonormal frame on $(M_1,g_1)$ and $\{\bar
e_\alpha\}_{\alpha=1}^{m_2}$ on
  $(M_2,g_2)$. Then $\{(e_j,0),(0,\frac{1}{\lambda} \bar e_\alpha)\}_{j=1,\ldots,m_1,\alpha=1,\ldots,m_2}$
forms a local orthonormal frame on $\overline{M}$. By the definition
of Ricci curvature, we have
    \begin{equation}\label{10-4}
     \begin{split}
    \overline{\mathrm{Ric} } \big( (X_1, X_2), (Y_1,Y_2) \big)
    & =\sum\limits_{i=1}^{m_1} \overline{\mathrm{Rm} } \big((e_i,0),(X_1, X_2),(e_i,0),
    (Y_1,Y_2)\big)\\
    &  + \sum\limits_{\alpha=1}^{m_2}\overline{\mathrm{Rm} }
       \big( (0,\frac{1}{\lambda} \bar e_\alpha),(X_1, X_2),(0,\frac{1}{\lambda} \bar
       e_\alpha).
       (Y_1,Y_2)\big)
       \end{split}
    \end{equation}
by substituting (\ref{6-30-5}) into (\ref{10-4}) and keeping in mind
the relation $g_2(\bar e_\alpha,X_2)g_2(\bar
e_\alpha,Y_2)=g_2(\sum_{\alpha}g_2(\bar e_\alpha,X_2)\bar
e_\alpha,Y_2)= g_2(X_2,Y_2)$, (\ref{6-30-6}) follows.

 Furthermore, since the scalar curvature $\overline{\mathrm{Scal}}$ satisfies
    \begin{equation}\label{10-4-0}
     \begin{split}
    \overline{\mathrm{Scal} }  =\sum\limits_{i=1}^{m_1} \overline{\mathrm{Ric} } \big((e_i,0),(e_i,0)\big)
    +\sum\limits_{\alpha=1}^{m_2}\overline{\mathrm{Ric} }
       \big( (0,\frac{1}{\lambda} \bar e_\alpha),(0,\frac{1}{\lambda} \bar e_\alpha)
       \big),
    \end{split}
    \end{equation}
substituting (\ref{6-30-6}) into (\ref{10-4-0}) gives
(\ref{6-30-7}).
\end{proof}

\begin{remark} It is not hard to verify that the three cases in Theorem (\ref{6-30-4}) agree with the
results in Propositions \ref{5-23-2} and \ref{5-23-3}. For instance,
by (\ref{6-30-6}), we have
    \begin{equation}\notag
     \begin{split}
     \overline{\mathrm{Ric} } \big( (0, V), (0,W) \big)
  &=\,^{M_2}\!\mathrm{Ric}(V,W)
  -\lambda g_2(V,W)\Delta_{M_1}\lambda
  -(m_2-1)\mid\mathrm{grad}\ \lambda \mid_{g_1}^2 g_2(V,W)\\
 &=\,^{M_2}\!\mathrm{Ric}(V,W)-\bigg(\frac{1}{\lambda}\Delta_{M_1}\lambda
          +\frac{m_2-1}{\lambda^2}\mid\mathrm{grad}\ \lambda \mid_{\bar g}^2
          \bigg) \bar g(V,W),
    \end{split}
     \end{equation}
 which is consistent with the third case (3) in Proposition \ref{5-23-3}.

\end{remark}

\begin{remark} Since (\ref{6-30-6}) and (\ref{6-30-7}) contain a term with factor $(m_2-1)$,
 to avoid trivial case, the dimension of $M_2$ is restrict to $m_2\geq 2$.
\end{remark}

\section{The behavior of warping function under Ricci flow }

 In this section, we shall use the unified version of Ricci curvature formula in the previous section
 to characterize the behavior of warping function under Ricci flow. More specifically, we wish to determine
 a certain condition which a smooth warping function satisfies such that the warped product metric is
 the solution to the corresponding Ricci flow.

 Before we launch the issue, let us state the
definition of the Ricci flow \cite{Ha, Br}.

\begin{definition} \label{5-15-1} Let $M$ be a manifold, and let $g(t),\ t\in [0,T)$, be
a one-parameter family of Riemannian metrics on $M$. We say that
$g(t)$ is a solution to the Rici flow if
  \begin{equation} \label{5-15-2}
  \frac{\partial}{\partial t}g(t)=-2Ric.
  \end{equation}
\end{definition}

For the warped product metrics, the Ricci flow is the evolution
equation
   \begin{equation} \label{9-28}
   \frac{\partial \bar g(x,y,t)}{\partial t}=-2\overline{\mathrm{Ric}}
\end{equation}
for a one-parameter family of Riemannian metrics $\bar g(t),\ t\in
[0,\bar T)$ on $\overline{M}$.

For behavior of the warping function on warped product manifold
under RF, we have the following main result.

\begin{theorem}\label{5-15-0} Suppose that Riemannian manifold $(M_1,g_1)$ is compact
(or complete non-compact) and $(M_2,g_2)$ is compact. Let
$(M_1,g_1(t))$ and $(M_2, g_2(t)$ be solutions to the Ricci flow on
a common time interval $[0, \bar T)$. Then the warped product metric
   $\bar g(t)=g_1(x,t)\oplus \lambda^2(x,t)g_2(y,t)$ is a solution to the
Ricci flow (\ref{9-28}) if and only if the warping function
$\lambda=\lambda(x,t), t\in [0,\bar T)$
 satisfies
   \begin{equation} \label{7-2-1}
   \begin{split}
  \frac{\partial \lambda(x,t)}{\partial t}
           -\left( 1+\frac{m_2}{m_1\lambda^2} \right)\Delta_{M_1}\lambda
     -\frac{m_2-1}{\lambda}\mid\mathrm{grad}\ \lambda \mid^2
    =\frac{\lambda^2-1}{m_2\lambda}\,^{M_2}\!\mathrm{Scal}
   \end{split}
   \end{equation}
 and
     \begin{equation} \label{7-2-1'}
       \mathrm{Hess}(\lambda)=0,
   \end{equation}
 where $m_i=dim M_i$.
 \end{theorem}

\begin{proof}
Since $g_i(t), i=1,2$ satisfy
     \begin{equation*}
   \begin{split}
    & \frac{\partial g_1(t)}{\partial t}=-2\,^{M_1}\!\mathrm{Ric}, \quad t\in [0,T_1),\\
    & \frac{\partial g_2(t)}{\partial t}=-2\,^{M_2}\!\mathrm{Ric},  \quad t\in  [0,T_2),
    \end{split}
  \end{equation*}
 we have the derivative of $\bar
g(t)$ with respect to the flow
 parameter $t$
   \begin{equation} \label{5-10-2}
   \begin{split}
   & \frac{\partial}{\partial t}( \bar g(t))\big( (X_1, X_2), (Y_1,Y_2) \big)\\
    &     =\Big( \frac{\partial g_1(t)}{\partial t}\oplus  \big(\lambda^2\frac{\partial g_2(t)}{\partial t}
          +\frac{\partial \lambda^2 }{\partial t} g_2(t)\big) \Big)\big( (X_1, X_2), (Y_1,Y_2) \big)\\
   & =-2\,^{M_1}\!\mathrm{Ric}(X_1,Y_1)-2\lambda^2\,^{M_2}\!\mathrm{Ric}(X_2, Y_2)
      +\frac{\partial \lambda^2 }{\partial t} g_2(t)(X_2,Y_2).
    \end{split}
     \end{equation}
Putting this with (\ref{6-30-6}) together, we can easily show that
$\bar g(x,y,t)$ is a solution to (\ref{9-28}), $\quad t\in [0,\bar
T=\min(T_1,T_2)\ )$ if and only if $\lambda(x,t)$ satisfies
 \begin{equation} \label{7-2-2}
   \begin{split}
   & \frac{\partial \lambda^2(x,t)}{\partial t}g_2(  X_2, Y_2 )
         =2(\lambda^2-1)\,^{M_2}\!\mathrm{Ric}(X_2,Y_2) +2\lambda g_2(X_2,Y_2) \Delta_{M_1} \lambda \\
   &+2\frac{m_2}{\lambda}\mathrm{Hess}(\lambda)(X_1,Y_1)+2(m_2-1)\mid\mathrm{grad}\ \lambda \mid^2 g_2(X_2,Y_2).
    \end{split}
     \end{equation}
On one hand, by the symmetry of $\,^{M_2}\!\mathrm{Ric}$, we can
choose an orthonormal basis $\{\bar e_\alpha\}$ on $M_2$ such that
$\,^{M_2}\!\mathrm{Ric}(\bar e_\alpha,\bar e_\beta)=0$ , $\alpha\neq
\beta$. Thus (\ref{7-2-2}) reduces to
          $$2\frac{m_2}{\lambda}\mathrm{Hess}(\lambda)(X_1,Y_1)=0,$$
which implies (\ref{7-2-1'}).

On the other hand,  by taking trace in both sides of (\ref{7-2-2})
with respect to $g_1$ and
  $g_2$, and noting that $\Delta_{M_1} \lambda=Tr_{g_1}\mathrm{Hess}(\lambda) $,
$\,^{M_2}\!\mathrm{Scal}=Tr_{g_2} \,^{M_2}\!\mathrm{Ric}$, we
conclude that (\ref{7-2-2}) is equivalent to
    \begin{equation} \notag
   \begin{split}
   & 2m_1 m_2 \lambda \frac{\partial \lambda}{\partial t}
         =2m_1(\lambda^2-1)\,^{M_2}\!\mathrm{Scal} +2\lambda m_1 m_2 \Delta_{M_1} \lambda \\
   &+2\frac{m_2^2}{\lambda}\Delta_{M_1} \lambda+2m_1 m_2(m_2-1)\mid\mathrm{grad}\ \lambda
   \mid^2,
    \end{split}
     \end{equation}
which implies (\ref{7-2-1}). Therefore we end the proof.
\end{proof}

\begin{remark}
(1) If $M_1$ is compact, then from (\ref{7-2-1}) and (\ref{7-2-1'})
we immediately see that $\lambda $ is a constant function in term to $M_1$.\\
(2) In Theorem \ref{5-15-0}, we don't stress that $\overline{M}$ is
compact or complete non-compact. Assume that $M_1$ is non-compact
complete manifold and $M_2$ is compact, then $\overline{M}$ is
complete non-compact. At this point we need to add a initial metric
$\bar g_0=(g_1)_0(x)\oplus \lambda^2(x,0)(g_2)_0(y)$ such that
      $\overline{ \mathrm{Riem} }_{\bar g_0} $ has a boundary.
\end{remark}

Now we concern about two questions: 1. Does the PDE (\ref{7-2-1})
have any solution? 2. How many degrees of freedom for the warping
function $\lambda$ are there?

In deed, it is easily seen that (\ref{7-2-1}) doesn't follow from
standard PDE theory. (\ref{7-2-1}) tells us that the terms on its
left-hand side only consist of  the points in the first factor
manifold $M_1$ and flow parameter $t$ whereas those on its
right-hand side consist of the points in the second factor manifold
$M_2$ besides $t$, thus one worries that such complicated nolinear
PDE (\ref{7-2-1}) may have no any solution $\lambda$.

As for degrees of freedom for the warping function $\lambda$, we
first consider the simplest cases: \\
(i) If $\lambda$ is constant, then from (\ref{7-2-1}) and
(\ref{7-2-1'}) we easily observe that $\lambda=\pm 1$. Since
$\lambda$ is positive, thus $\lambda=1$, which implies that
$\overline{M}$ is exactly a direct
product manifold. This is a true. \\
(ii) Assume $M_2$ has constant scalar curvature, then (\ref{7-2-1})
no longer involves the point of $M_2$. This should be a kernel heat
equation, of course it must have solution.

The next two theorems naturally give a guarantee for existence of
solution to (\ref{7-2-1}) as long as there exists a warped product
solution $\bar g(t)$ to the RF. From the short-time existence and
uniqueness result for Ricci flow on a compact manifold \cite{Ha,
De}, we give the corresponding version for WPM.

\begin{theorem} Let $\big( M_1\times M_2,
  (\bar g_0)_{ij\alpha\beta}(x,y):=(g_1)^0_{ij}(x,t)+\lambda^2(x) (g_2)^0_{\alpha\beta}(y,t) \big)$
be a compact Riemannian manifold. Then there exists a constant $\bar
T>0$ such that the initial value problem
  $$\left\{\begin{array}{l}
\frac{\partial}{\partial t}(\bar g_{ij\alpha\beta}(x,y, t))
        =-2 \overline{ \mathrm{Ric}} _{ij\alpha\beta}(x,y,t)\\
  \bar g_{ij\alpha\beta}(x,y,0)= (\bar g_0)_{ij\alpha\beta}(x,y)
  \end{array}\right.
  $$
has a unique smooth solution $\bar
g_{ij\alpha\beta}(x,y,t)=(g_1)_{ij}(x,t)\oplus \lambda^2(x,t)
(g_2)_{\alpha\beta}(y,t)$ on $\overline{M} \times [0, \bar T)$,
where $\overline{ \mathrm{Ric}}
_{ij\alpha\beta}(x,y,t):=\,^{M_1}\!\mathrm{Ric}_{ij}(x,t)+\lambda^2(x,t)
\,^{M_2}\!\mathrm{Ric}_{\alpha\beta}(y,t)$.
\end{theorem}

On a non-compact complete manifold $\overline{M}$, we only require
the short-time existence established by Shi \cite{Sh}. The following
result is modified to the warped product case according the version
of Shi.

\begin{theorem} \label{10-1-1} Let $\big( M_1\times M_2 , \bar g_0(x, y)=(g_1)^0(x) \oplus
\lambda_0^2(x) (g_2)^0(y) \big)$ be a complete noncompact Riemannian
manifold of dimension $m_1+m_2$ with bounded curvature. Then there
exists a constant $\bar T>0$ such that the initial value problem
 $$\left\{\begin{array}{l}
 \frac{\partial}{\partial t}(\bar g_{ij\alpha\beta}(x,y, t))
        =-2 \overline{ \mathrm{Ric}} _{ij\alpha\beta}(x,y,t),\\
  \bar g_{ij\alpha\beta}(x,y,0)=(g_1)^0_{ij}(x)\oplus
  \lambda^2_0(x)(g_2)^0_{\alpha\beta}(y)
  \end{array}\right.
  $$
has a smooth solution $\bar
g_{ij\alpha\beta}(x,y,t)=(g_1)_{ij}(x,t)\oplus \lambda^2(x,t)
(g_2)_{\alpha\beta}(y,t)$ on $\overline{M} \times [0, \bar T]$ with
uniformly bounded curvature.
\end{theorem}

Now we construct a relatively simple example.
\begin{example} \label{9-30-2} Let $M_1=\mathbb R$ with flat metric $g_1=h(x)=\mu^2(x) dx^2$
( $\mu(x)$ is a smooth positive function ) and $M_2=S^{n}$ $(n \geq
2)$ with the standard metric which implies $M_2$ admits an Einstein
metric $g_2=\lambda^2(x)g_{S^n}$. By the main result in \cite{Si},
under some constraints for initial values, there exists warping
functions $\lambda(x,t)$ and a maximal constant $T$ such that warped
product solution
  $$\bar g(x,y)=h(x,t)\oplus \lambda^2(x,t)g_{S^n}(y),\quad t\in [0,T)$$
  to the RF (\ref{9-28}). Of course, we don't write $\lambda(x,t)$ as explicit form.
  On $\overline M=\mathbb R\times S^n$, the warped product metric
  $\bar g=\mu^2(x) dx^2\oplus \lambda^2(x)g_{S^n}$ can be read as
      $$\bar g(s,y)=ds^2 \oplus \lambda^2(s)g_{ S^{n} }(y),$$
where $s=\int_0^x \mu(x) dx$ is the arc-length parameter. Then the
sectional curvatures of planes containing or perpendicular to the
radical vector $\frac{\partial }{\partial s}=\frac{1}{\mu(x)}
\frac{\partial}{\partial x}$ are respectively ( cf. Chap.3 in
\cite{Pe}, or \cite{AK,MX} )
   $$ K_{rad}=-\frac{\lambda_{ss}}{\lambda},\quad  K_{sph}=\frac{1-\lambda_s^2}{\lambda^2},$$
and the Ricci tensor is
    \begin{equation}\label{10-2-01}
     \begin{split}
      \overline{\mathrm{Ric}}&=-n\frac{\mu
      \lambda_{xx}-\lambda_x\mu_x}{\lambda\mu}dx^2
       \oplus \left(
       -\frac{\lambda\mu\lambda_{xx}+(n-1)\mu\lambda^2_x-\lambda\lambda_x\mu_x}{\mu^3}+n-1\right)g_{S^n}\\
      & =-n\frac{\lambda_{ss}}{\lambda}ds^2
         \oplus \bigg( (n-1)\big(1-\lambda_s^2 \big)-\lambda\lambda_{ss}\bigg)g_{S^n}\\
          &=n K_{rad}ds^2\oplus \left( K_{rad}+(n-1)K_{sph}\right)\lambda^2(s)g_{ S^{n}  }.
          \end{split}
    \end{equation}
Since $dx^2$ and $g_{S^n}$ are independent of $t$, a direct
computation gives
       \begin{equation}\label{10-2-02}
     \begin{split}
     &\frac{\partial }{\partial t}\big( \mu^2(x,t) dx^2\oplus \lambda^2(x,t)g_{S^n}(y) \big)\\
     &= 2\mu\mu_tdx^2 \oplus 2\lambda\lambda_t g_{S^n}\\
     &=2\frac{\mu_t}{\mu}ds^2\oplus 2\lambda\lambda_t g_{S^n}.
      \end{split}
    \end{equation}
 Hence if the warped product metrics
$\bar g(x,y,t)=\mu^2(x,t) dx^2\oplus \lambda^2(x,t)g_{S^n}(y)$ is a
solution to the Ricci flow (\ref{9-28}), then substituting
(\ref{10-2-01}) and (\ref{10-2-02}) into (\ref{9-28}) immediately
yields
     \begin{equation} \label{9-30-0}
     \left\{  \begin{array}{l}
    \frac{ \lambda_t}{\lambda}
        =-\big(K_{rad}+(n-1)K_{sph}\big),\\
     \frac{\mu_t}{\mu}=-nK_{rad},
    \end{array} \right.
    \end{equation}
which happens to be
    \begin{equation} \label{9-30}
     \left\{  \begin{array}{l}
    \frac{\partial \log \lambda}{\partial t}
        =-\big(K_{rad}+(n-1)K_{sph}\big),\\
     \frac{\partial \log\mu}{\partial t}=-nK_{rad}.
    \end{array} \right.
    \end{equation}
 Since the sectional curvature functions $K_{rad}$ and
$K_{sph}$ are uniform bound ( see the proof of Theorem 1.2 in
\cite{MX} ), we integrate (\ref{9-30}) over the time interval
$[0,t],t<T$ and get the  functions
   \begin{equation*}
    \left\{  \begin{array}{l}
   \lambda(x,t)=\lambda(x,0)e^{-\int_0^t\big(K_{rad}+(n-1)K_{sph}\big)dt},\\
   \mu(x,t)=\mu(x,0) e^{-n\int_0^tK_{rad} dt}.
   \end{array} \right.
    \end{equation*}
\end{example}

\section{The behavior of warping function under HGF}

 We now investigate the behavior of warping function under the
 hyperbolic geometric flow.

 Recall that Kong and Liu \cite{KL} introduced a geometric flow called
hyperbolic geometric flow (HGF) whose definition is as follows.

\begin{definition} \label{5-24-0} Let $M$ be a Riemannian manifold. The hyperbolic geometric flow (HGF) is the
evolution equation
  \begin{equation} \label{5-24-1}
  \frac{\partial^2}{\partial t^2}g(t)=-2Ric
  \end{equation}
 for a one-parameter family of Riemannian metrics $g(t),\ t\in [0,T)$ on $M$. We say that $g(t)$ is a
  solution to the hyperbolic geometric flow if it satisfies (\ref{5-24-1}).
\end{definition}

When $M$ is changed to our warped product manifold $\overline{M}$,
the corresponding HGF is
     \begin{equation} \label{10-1}
  \frac{\partial^2}{\partial t^2}\bar g(t)=-2 \overline{\mathrm{Ric} }.
  \end{equation}
In this case, similar to Theorem \ref{5-15-0} we have

\begin{theorem} \label{5-24-2} Suppose that Riemannian manifold $(M_1,g)$ is compact (or complete
non-compact) and $M_2$ is compact. If $(M_1,g_1(t))$ and $(M_2,
g_2(t)$ are the solution to the HGF on a common time interval $I$,
respectively, then the warped product metric $\bar
g(x,y,t)=g_1(x,t)\oplus \lambda^2(x,t)g_2(y,t)$ is a solution to the
HGF (\ref{10-1}) if
 and only if the warped product function $\lambda=\lambda(x,t),t\in I$
 satisfies
   \begin{equation} \label{7-2-6}
   \begin{split}
 & \frac{ m_2}{2} \frac{\partial^2 \lambda^2}{\partial t^2}
  -\frac{(\lambda^2+m_2)m_2}{\lambda} \Delta_{M_1} \lambda
   -m_2(m_2-1)\mid\mathrm{grad}\ \lambda \mid^2\\
  &=(\lambda^2-1)\,^{M_2}\!\mathrm{Scal}
   -Tr_{g_2}\big(\frac{\partial g_2}{\partial t}\big) \frac{\partial \lambda^2}{\partial t}
 \end{split}
   \end{equation}
 and
     \begin{equation} \label{7-2-6'}
   m_1\frac{\partial \lambda^2 }{\partial t}\frac{\partial g_2(t)}{\partial
    t}(\bar e_\alpha, \bar e_\beta)+  \frac{m_2}{\lambda}
    \Delta_{M_1} \lambda=0, \quad \alpha\neq \beta,
   \end{equation}
 where $\{\bar e_{\alpha} \}$ is an orthonormal basis on $M_2$ such
 that $\,^{M_2}\!\mathrm{Ric}(\bar e_\alpha, \bar e_\beta) =0$.
 \end{theorem}

\begin{proof} Since $g_i(t), i=1,2$ satisfy
     \begin{equation*}
   \begin{split}
    & \frac{\partial^2 g_1(t)}{\partial t^2}=-2\,^{M_1}\!\mathrm{Ric},\\
    & \frac{\partial^2 g_2(t)}{\partial
    t^2}=-2\,^{M_2}\!\mathrm{Ric},\quad t\in I,
    \end{split}
  \end{equation*}
 we have
   \begin{equation} \notag
   \begin{split}
   & \frac{\partial^2\bar g(t)}{\partial t^2}\big( (X_1, X_2), (Y_1,Y_2) \big)\\
    &     = \frac{\partial^2 g_1(t)}{\partial t^2}(X_1,Y_1)
       +\lambda^2\frac{\partial^2 g_2(t)}{\partial t^2}(X_2,Y_2)\\
     &+2\frac{\partial \lambda^2 }{\partial t}\frac{\partial g_2(t)}{\partial t}(X_2,Y_2)
      + \frac{\partial^2 \lambda^2 }{\partial t^2} g_2(t)(X_2,Y_2)\\
   & =-2\,^{M_1}\!\mathrm{Ric}(X_1,Y_1)-2\lambda^2\,^{M_2}\!\mathrm{Ric}(X_2, Y_2)\\
  &+2\frac{\partial \lambda^2 }{\partial t}\frac{\partial g_2(t)}{\partial t}(X_2,Y_2)
      + \frac{\partial^2 \lambda^2 }{\partial t^2} g_2(t)(X_2,Y_2).\\
    \end{split}
     \end{equation}
Combining this and (\ref{6-30-6}), we easily see that $\bar
g(x,y,t)$ is the solution to (\ref{10-1}) if and only if
$\lambda=\lambda(x,y,t)$ satisfies
   \begin{equation} \label{7-2-7}
   \begin{split}
  & \frac{\partial^2 \lambda^2 }{\partial t^2} g_2(t)(X_2,Y_2)
    +2\frac{\partial \lambda^2 }{\partial t}\frac{\partial g_2(t)}{\partial
    t}(X_2,Y_2)\\
  & =(2\lambda^2-2)\,^{M_2}\!\mathrm{Ric}(X_2, Y_2)+2\lambda \Delta_{M_1} \lambda g_2(X_2,Y_2)\\
  & +\frac{2m_2}{\lambda} \mathrm{Hess}(\lambda)(X_1,Y_1)
    +2(m_2-1)\mid\mathrm{grad}\lambda \mid^2 g_2(X_2,Y_2).
    \end{split}
     \end{equation}
After we choose an orthonormal basis $\{\bar e_\alpha\}$ on $M_2$
such that $\,^{M_2}\!\mathrm{Ric}(\bar e_\alpha,\bar e_\beta)=0$ ,
$\alpha\neq \beta$,  (\ref{7-2-7}) reduces to
         \begin{equation} \label{7-2-7'}
   \frac{\partial \lambda^2 }{\partial t}\frac{\partial g_2(t)}{\partial
    t}(\bar e_\alpha,\bar e_\beta) =\frac{m_2}{\lambda} \mathrm{Hess}(\lambda)(X_1,Y_1).
         \end{equation}
Further trace it with respect to $g_1$, we get
      $$m_1\frac{\partial \lambda^2 }{\partial t}\frac{\partial g_2(t)}{\partial
    t}(\bar e_\alpha,\bar e_\beta)=\frac{m_2}{\lambda}\Delta_{M_1} \lambda,\quad \alpha\neq \beta,$$
 which is just (\ref{7-2-6'}).

On the other hand,  by taking trace in both sides of (\ref{7-2-7})
with respect to $g_1$ and $g_2$, we can  reduce (\ref{7-2-7}) to
    \begin{equation} \notag
   \begin{split}
 &  m_1 m_2  \frac{\partial^2 \lambda^2}{\partial t^2}
         +2m_1 \frac{\partial \lambda^2 }{\partial t} Tr_{g_2} \big(\frac{\partial g_2(t)}{\partial
    t}\big)\\
&=2m_1 (\lambda^2-1)\,^{M_2}\!\mathrm{Scal} +2\lambda m_1 m_2 \Delta_{M_1} \lambda \\
   &+\frac{2m_1m_2^2}{\lambda}\Delta_{M_1} \lambda+2m_1 m_2(m_2-1)\mid\mathrm{grad}\
   \lambda\mid^2,
    \end{split}
     \end{equation}
which implies (\ref{7-2-6}).
\end{proof}

Obviously, the equation (\ref{7-2-6}) is analogous to the previous
equation (\ref{7-2-1}) but much more complicated, manifested chiefly
by the second-order derivative term $ \frac{\partial^2
\lambda^2}{\partial t^2}$ and the extra term
 $Tr_{g_2}\big(\frac{\partial g_2(t)}{\partial t}\big)$
 without carrying given information. Therefore one may worry about the equation (\ref{7-2-6})
 has no any solution and makes no any sense. This need not worry, because the short-time existence
 result for HGF on a compact manifold (see Theorem 1.1 in \cite{DKL}) can provide us
an evidence. We give its version related to WPM as follows.

\begin{theorem} \label{10-1-2} Let $\big(M_1\times M_2 , \bar g^0(x, y)=(g_1)^0(x) \oplus
\lambda_0^2(x) (g_2)^0(y) \big)$ be a compact Riemannian manifold.
Then there exists a constant $\bar T>0$ such that the initial value
problem
  $$\left\{\begin{array}{l}
    \frac{\partial^2}{\partial t^2}(\bar g_{ij\alpha\beta}(x,y, t))
         =-2 \overline{ \mathrm{Ric}} _{ij\alpha\beta}(x,y,t)\\
  \bar g_{ij\alpha\beta}(x,y,0)=\bar g^0_{ij\alpha\beta}(x,y),\quad
   \frac{\partial}{\partial t} \bar g_{ij\alpha\beta}(x,y,0)=h^0_{ij\alpha\beta}(x,y)
  \end{array}\right.
  $$
has a unique smooth solution
 $\bar g_{ij\alpha\beta}(x,y,t)=(g_1)_{ij}(x,t)\oplus \lambda^2 (g_2)_{\alpha\beta}(y,t)$
 on $\overline M \times [0, \bar T]$, where $h^0_{ij\alpha\beta}(x,y)$ is a symmetric tensor on
$M_1\times M_2$.
\end{theorem}

In non-compact complete manifold $\overline{M}$, framing Theorem
\ref{10-1-1} and Theorem \ref{10-1-2} and combining Theorem 3.1 in
\cite{Si} (see Introduction section), we present an analogous result
ro Theorem \ref{10-1-1} without proof.

\begin{proposition}  Let
$\big(M_1\times M_2, \bar g_0(x, y)=(g_1)^0(x) \oplus \lambda_0^2(x)
(g_2)^0(y) \big)$ and $\big( M_1\times M_2, \bar h_0(x, y) \big)$ be
complete noncompact Riemannian manifolds of dimension $m_1+m_2$ with
bounded curvature and $\lambda_0$ be imposed certain constrains.
Then there exists a constant $\bar T>0$ such that the initial value
problem
 $$\left\{\begin{array}{l}
 \frac{\partial}{\partial t}(\bar g_{ij\alpha\beta}(x,y, t))
        =-2 \overline{ \mathrm{Ric}} _{ij\alpha\beta}(x,y,t),\\
  \bar g_{ij\alpha\beta}(x,y,0)=\bar g^0_{ij\alpha\beta}(x,y)
    =(g_1)^0_{ij}(x)\oplus \lambda^2_0(x)(g_2)^0_{\alpha\beta}(y),\\
    \frac{\partial}{\partial t}\bar g_{ij\alpha\beta}(x,y,0)=h^0_{ij\alpha\beta}(x,y)
  \end{array}\right.
  $$
has a smooth solution $\bar
g_{ij\alpha\beta}(x,y,t)=(g_1)_{ij}(x,t)\oplus \lambda^2
(g_2)_{\alpha\beta}(y,t)$ on $\overline{M} \times [0, \bar T]$ with
uniformly bounded curvature.
\end{proposition}

In order to gain a sense for (\ref{7-2-6}), we present several
special examples.

\begin{example} (Trivial example) If $\lambda$ is constant, then we easily observe
from (\ref{7-2-6}) and (\ref{7-2-6'}) that $\lambda=\pm 1$. Since
$\lambda$ is positive, thus $\lambda=1$, which implies that
$\overline{M}$ is exactly a direct product manifold. This is a fact.
\end{example}

\begin{example} \label{7-2-07} For simplicity sake, we manage to let the unknown term
$\frac{\partial }{\partial t} \bar g_2(x,y,t)=0$ in (\ref{7-2-6}).
Take $M_2=S^{n}$ $(n \geq 2)$ with the standard metric which implies
$M_2$ admits an Einstein metric $g_2=g_{S^n}$. Like the previous
Example \ref{9-30-2}, let $M_1=\mathbb R$ with flat metric
$g_1=\mu(x)dx^2$. On $\overline M=\mathbb R\times S^n$, the warped
product metric
  $\bar g=\mu^2(x) dx^2\oplus \lambda^2(x)g_{S^n}$ can be read as
      $$\bar g(s,y)=ds^2 \oplus \lambda^2(s)g_{ S^{n} }(y),$$
where $s=\int_0^x \mu(x) dx$ is the arc-length parameter.

Remembering $dx^2$ and $g_{S^n}$ are independent of $t$, we get
       \begin{equation}\label{10-2-03}
     \begin{split}
     &\frac{\partial^2 }{\partial t^2}\big( \mu^2(x,t) dx^2\oplus \lambda^2(x,t)g_{S^n}(y) \big)\\
     &= 2(\mu\mu_{tt}+\mu_t^2)dx^2 \oplus 2(\lambda\lambda_{tt}+\lambda_t^2) g_{S^n}\\
     &=2\frac{\mu\mu_{tt}+\mu_t^2}{\mu^2}ds^2\oplus 2(\lambda\lambda_{tt}+\lambda_t^2) g_{S^n}.
      \end{split}
    \end{equation}
Therefore, if the warped product metrics $\bar g(x,y,t)=\mu^2(x,t)
dx^2\oplus \lambda^2(x,t)g_{S^n}(y)$ is a solution to the HGF
(\ref{10-1}), then substituting (\ref{6-30-6}) and (\ref{10-2-03})
into (\ref{10-1}), we obtain
  \begin{equation} \label{10-2-2}
    \left\{  \begin{array}{l}
     \frac{ \lambda\lambda_{tt}+\lambda^2_t}{\lambda^2}=-\big(K_{rad}+(n-1)K_{sph}\big),\\
     \frac{\mu\mu_{tt}+\mu_t^2}{\mu^2}=-nK_{rad},
     \end{array} \right.
     \end{equation}
which happens to be
     \begin{equation} \label{10-2-3}
     \left\{  \begin{array}{l}
     \frac{\partial^2 \bar\lambda}{\partial t^2}
        =-\big(K_{rad}+(n-1)K_{sph}\big),\\
     \frac{\partial^2 \bar \mu}{\partial t^2}=-nK_{rad},
    \end{array} \right.
    \end{equation}
where we assume that there are exactly the relations
  \begin{equation} \label{10-2-4}
   \bar\lambda_{tt}=\frac{\lambda\lambda_{tt}+\lambda^2_t}{\lambda^2}
   \end{equation}
and
   \begin{equation} \label{10-2-5}
     \bar \mu_{tt}=\frac{\mu\mu_{tt}+\mu_t^2}{\mu^2}.
   \end{equation}
 Since the sectional curvature functions $K_{rad}$ and
$K_{sph}$ are of uniform bound, we can integrate (\ref{10-2-3}) over
the time interval $[0,t],t<T$ for twice and get
   \begin{equation*}
    \left\{  \begin{array}{l}
   \bar \lambda(x,t)=\bar \lambda(x,0)+t \bar\lambda_t(x,0)
           -\int_0^t\int_0^u \big(K_{rad}+(n-1)K_{sph}\big)dudt,\\
    \bar \mu(x,t)=\bar\mu(x,0)+t\bar\mu_t(x,0)-n \int_0^t\int_0^u K_{rad}dudt.
   \end{array} \right.
    \end{equation*}
Further we locally re-solve the original functions $\lambda(x,t)$
and $\mu(x,t)$.
\end{example}

\begin{remark}
 Although (\ref{10-2-2}) may look simple and has a local solution, we have to remind it
 is a set of nonlinear weakly hyperbolic PDEs
   \begin{equation} \label{10-2-3'}
     \left\{  \begin{array}{l}
     \lambda_{tt}-\lambda_{ss}=
     \frac{1}{\lambda}\lambda_t^2+\frac{1}{\lambda}\lambda_s^2-(n-1)\lambda,\\
       \frac{1} \mu_{tt}+\frac{1}{\mu^2}\mu_t^2=\frac{n}{\lambda}\lambda_{ss},
      \end{array} \right.
    \end{equation}
 which is almost never easy to solve. (\ref{10-2-4}) and
 (\ref{10-2-5}) may be only our own wishful thinking or be taken for granted.
\end{remark}

\section{Evolution equations of warping function and Ricci curvature  }

 In this section we present evolution equations for
an arbitrary family of warped product metrics $\bar
g(x,y,t)=g_1(x,t)\oplus \lambda^2(x,t)g_2(y,t), t\in [0,T]$ with
Einstein metric $g_2$ that is evolving by RF and by HGF as well. We
also present evolution equations for the Ricci curvatures of such an
 evolving metric. Our idea mainly comes from Simon's strategy \cite{Si}.

From now on, we make informal convention for some notations on
$\overline M=M_1\times_{\lambda} M_2$:
  \begin{equation}  \label{10-5-00}
   \begin{split}
 &\partial_i :=\frac{\partial}{\partial x^i}, \quad i=1,\ldots,m_1;\quad
 \partial_\alpha :=\frac{\partial}{\partial y^\alpha}, \quad \alpha=1,\ldots, m_2;\\
 & \bar g_{ij}=\bar g_{(i0)(j0)}:={\bar g}\big((\partial_i,0),(\partial_j,0) \big);
 \quad \bar g_{\alpha\beta}=\bar g_{(0\alpha)(0\beta)}:=\bar g\big( (0,\partial_\alpha),(0,\partial_\beta)\big);\\
  &\bar g_{ij\alpha\beta}=\bar g_{(i\alpha)(j\beta)}
                :=\bar g
                \big((\partial_i,\partial_\alpha),(j,\partial_\beta)\big);\quad
         ( \bar g^{ij\alpha\beta}):=(\bar g_{ij\alpha\beta})^{-1};\\
  &    \overline{\mathrm{Rm}}_{(i\alpha)(j\beta)(k\sigma)(l\tau)}:=
         \overline{\mathrm{Rm}}\left( (\partial_i,\partial_\alpha),(\partial_j,\partial_\beta),
         (\partial_k,\partial_\sigma),(\partial_l,\partial_\tau)\right);\\
  &\overline{\mathrm{Ric}}_{ij}=\overline{\mathrm{Ric}}_{(i0)(j0)}
     :=\overline{\mathrm{Ric}}\big((\partial_i,0),(\partial_j,0)\big);
     \quad
  \overline{\mathrm{Ric}}_{ij\alpha\beta}=\overline{\mathrm{Ric} }_{(i\alpha)(j\beta)}
     :=\overline{\mathrm{Ric}}\big((\partial_i,\partial_\alpha),(\partial_j,\partial_\beta)\big).
    \end{split}
   \end{equation}

\subsection{Metric and warping function evolution equations}

Since the cross terms of $\bar g$ are zero, we need only to consider
the evolution equations of $\bar g_{ij}$ and $\bar g_{\alpha\beta}$.
Meanwhile we make a assumption that $g_2$ has a fixed Einstein
metric of the form $\,^{M_2}\!\mathrm{Ric}=c g_2$ (c is some
constant) and derive the evolution equation of warping function.

\begin{proposition}\label{10-5-01} Let the smooth warped product metric
      $$\bar g(x,y,t)=g_1(x,t)\oplus \lambda^2(x,t)g_2(y,t), t\in [0,\bar T)$$
 be a solution to the Ricci flow (\ref{9-28}) on the manifold $M_1\times M_2$.
Then the metrics $g_1$ and $g_2$ satisfy the evolution equations\\
  \begin{equation} \label{10-5-02}
      \frac{\partial}{\partial t}(g_1)_{ij} =-2
      \,^{M_1}\!\mathrm{Ric}_{ij}
          +\frac{2m_2}{\lambda} \mathrm{Hess}(\lambda)(\partial_i,\partial_j),\\
   \end{equation}

  \begin{equation} \label{10-5-03}
      \frac{\partial}{\partial t} \left(\lambda^2 (g_2)_{\alpha\beta}\right)=-2
      \,^{M_2}\!\mathrm{Ric}_{\alpha\beta}
          +\left(\Delta_{M_1}\lambda^2 +(2m_2-4)|\mathrm{grad}
 \lambda|^2\right)(g_2)_{\alpha\beta}
  \end{equation}
\end{proposition}

\begin{proof} Since we see that $(g_1)_{ij}=\bar g_{ij}$ and
  $\lambda^2 (g_2)_{\alpha\beta}=\bar g_{\alpha\beta}$, by using the
  Ricci flow (\ref{9-28}) and Ricci curvature formula
  (\ref{6-30-6}), we immediately get the desired identities
  (\ref{10-5-02}) and (\ref{10-5-03}).
\end{proof}

\begin{corollary}  Soppose that $g_2$ is a fixed Einstein metric with constant $c$. Then under
the Ricci flow (\ref{9-28}), the warping function $\lambda$
satisfies the following evolution equation

  \begin{equation} \label{10-5-04}
      \frac{\partial}{\partial t} \lambda^2 =-2c
           +\Delta_{M_1}\lambda^2 +(2m_2-4)|\mathrm{grad} \lambda|^2
  \end{equation}
\end{corollary}

\begin{proof} By already assumption, $g_2$ is independent of $t$.
Combining this and $\,^{M_2}\!\mathrm{Ric}_{\alpha\beta}
  =c(g_2)_{\alpha\beta}$, (\ref{10-5-04}) follows from (\ref{10-5-03}).
\end{proof}

Similar to the above results, we have the parallel conclusions under
the HGF.

\begin{proposition}\label{10-5-05} Let
      $$\bar g(x,y,t)=g_1(x,t)\oplus \lambda^2(x,t)g_2(y,t), t\in [0,\bar T)$$
 be a solution to the hyperbolic geometric flow (\ref{10-1}) on the manifold $M_1\times M_2$, where
 $g_2$ is a fixed Einstein metric with constant $c$.
Then the metrics $g_1$ and the warping function $\lambda$ satisfy the evolution equation\\
  \begin{equation} \label{10-5-06}
      \frac{\partial^2}{\partial t^2}(g_1)_{ij} =-2
      \,^{M_1}\!\mathrm{Ric}_{ij}
          +\frac{2m_2}{\lambda} \mathrm{Hess}(\lambda)(\partial_i,\partial_j),\\
   \end{equation}

\begin{equation} \label{10-5-07}
      \frac{\partial^2}{\partial t^2} \lambda^2 =-2c
           +\Delta_{M_1}\lambda^2 +(2m_2-4)|\mathrm{grad} \lambda|^2
    \end{equation}
\end{proposition}

\subsection{Ricci curvature evolution equations}

From (\ref{6-30-6}) or Proposition \ref{5-23-3}, we see that the
cross terms of $\overline{\mathrm{Ric}}$ are zero. Hence we only
consider the evolution equations for $\overline{\mathrm{Ric}}_{ij}$
and $\overline{\mathrm{Ric}}_{\alpha\beta}$.

\begin{theorem}\label{10-5-07} Under the Ricci flow (\ref{9-28})
on the manifold $\overline M$, the Ricci curvature
$\overline{\mathrm{Ric}}_{ij}$ and
$\overline{\mathrm{Ric}}_{\alpha\beta}$ satisfy the following evolution equations\\
  \begin{equation} \label{10-5-08}
      \frac{\partial }{\partial t} \overline{\mathrm{Ric}}_{ij}
     = \bar\Delta \overline{\mathrm{Ric}}_{ij}
       +\frac{2}{m_2}\bar g^{\alpha\beta}\overline{\mathrm{Ric}}_{\alpha\beta}
           \big(\overline{\mathrm{Ric}}_{ij}-\,^{M_1}\!\mathrm{Ric}_{ij}\big)
    -2\bar g^{kl}\overline{\mathrm{Ric}}_{ik}\overline{\mathrm{Ric}}_{jl},
   \end{equation}

  \begin{equation} \label{10-5-09}
      \begin{split}
    &\frac{\partial }{\partial t} \overline{\mathrm{Ric}}_{\alpha\beta}
     = \bar\Delta \overline{\mathrm{Ric}}_{\alpha\beta}
      +\frac{2}{m_2}\bar g^{k l} \bar g^{p q } \overline{\mathrm{Ric}}_{l q }
      \big(\overline{\mathrm{Ric}}_{kp}-\,^{M_1}\!\mathrm{Ric}_{kp}\big)\bar g_{\alpha\beta}
      -2\bar g^{ \gamma \delta}\overline{\mathrm{Ric}}_{\gamma \alpha}
         \overline{\mathrm{Ric}}_{\delta \beta }\\
       &\quad +2\lambda^2\bar g^{\gamma \delta} \bar g^{\sigma \tau}
        \overline{\mathrm{Ric}}_{\delta \tau}\left(\,^{M_2}\!\mathrm{Rm}_{\alpha\gamma\beta\sigma}
    +|\mathrm{grad} \lambda|^2 \big( (g_2)_{\alpha\sigma}(g_2)_{\beta\gamma}
      -(g_2)_{\alpha\beta}(g_2)_{\gamma\sigma} \big)  \right).\\
      \end{split}
  \end{equation}
\end{theorem}

\begin{proof} According to our notational convention (\ref{10-5-00}) on
$\overline M$, the evolution equation of the Ricci curvature in
\cite{Ha} is transformed into such form as
    \begin{equation} \label{10-5-10}
    \frac{\partial }{\partial t} \overline{\mathrm{Ric}}_{\bar\i \bar\j \bar\alpha \bar\beta}
     = \bar\Delta \overline{\mathrm{Ric}}_{\bar\i \bar\j \bar\alpha \bar\beta}
       +2\bar g^{\bar k\bar l\bar \gamma\bar \delta} \bar g^{\bar p\bar q \bar\sigma \bar\tau}
       \overline{\mathrm{Rm}}_{(\bar k \bar \gamma)(\bar\i \bar\alpha)(\bar p \bar\sigma)(\bar\j \bar\beta)}
       \overline{\mathrm{Ric}}_{\bar l \bar q \bar\delta \bar\tau}
 -2\bar g^{\bar k\bar l\bar \gamma \bar \delta}\overline{\mathrm{Ric}}_{\bar k \bar \i \bar\gamma \bar\alpha }
     \overline{\mathrm{Ric}}_{\bar l \bar \j \bar\delta \bar\beta }.
     \end{equation}
 where $\bar\i,\bar\j=0,1,\ldots,m_1$,\quad $\bar\alpha,\bar\beta=0,1,\ldots,m_2$, etc. Hence we have
      \begin{equation} \label{10-5-11}
    \frac{\partial }{\partial t} \overline{\mathrm{Ric}}_{ij}
     = \bar\Delta \overline{\mathrm{Ric}}_{ij}
       +2\bar g^{\bar k\bar l\bar \gamma\bar \delta} \bar g^{\bar p\bar q \bar\sigma \bar\tau}
       \overline{\mathrm{Rm}}_{(\bar k \bar \gamma)(i0)(\bar p \bar\sigma)(j0)}
       \overline{\mathrm{Ric}}_{\bar l \bar q \bar\delta \bar\tau}
 -2\bar g^{\bar k\bar l\bar \gamma \bar \delta}\overline{\mathrm{Ric}}_{\bar k i \bar\gamma 0 }
     \overline{\mathrm{Ric}}_{\bar l j \bar\delta 0 }.
     \end{equation}
(\ref{6-30-3}) and (\ref{6-30-5}) tell us that the only non-zero
  $\overline{\mathrm{Rm}}_{(\bar k \bar \gamma)(i0)(\bar p\bar\sigma)(j0)}$ are of the form
  $\overline{\mathrm{Rm}}_{(0 \gamma)(i0)(0 \sigma)(j0)}$. On the other hand,
  we also see that $\bar g^{i00\alpha}=\bar g^{i\alpha}=0$ and
  $\overline{\mathrm{Ric}}_{0i\alpha 0}=\overline{\mathrm{Ric}}_{\alpha
  i}=0$. Putting these facts together, (\ref{10-5-11}) can be reduced to
     \begin{equation} \label{10-5-12}
    \frac{\partial }{\partial t} \overline{\mathrm{Ric}}_{ij}
     = \bar\Delta \overline{\mathrm{Ric}}_{ij}
       +2\bar g^{\gamma \delta} \bar g^{\sigma \tau}
       \overline{\mathrm{Rm}}_{(0 \gamma)(i0)(0\sigma)(j0)}
       \overline{\mathrm{Ric}}_{\delta \tau}
 -2\bar g^{kl}\overline{\mathrm{Ric}}_{ki}\overline{\mathrm{Ric}}_{lj}.
     \end{equation}
Since (\ref{6-30-5}) gives
    \begin{equation}\label{10-6-1}
     \overline{\mathrm{Rm}}_{(0 \gamma)(i0)(0\sigma)(j0)}
      =-\lambda \mathrm{Hess}(\lambda)(\partial_i,\partial_j) (g_2)_{\gamma\sigma}
      =-\frac{1}{\lambda} \mathrm{Hess}(\lambda)(\partial_i,\partial_j) \bar g_{\gamma\sigma},
     \end{equation}
again (\ref{6-30-6}) gives
     \begin{equation}\label{10-6-2}
     \overline{\mathrm{Ric}}_{ij}
      =\,^{M_1}\!\mathrm{Ric}_{ij}-\frac{m_2}{\lambda}\mathrm{Hess}(\lambda)(\partial_i,\partial_j),
      \end{equation}
combining (\ref{10-6-1}) and (\ref{10-6-2}) gives
       \begin{equation}\label{10-6-3}
   \overline{\mathrm{Rm}}_{(0 \gamma)(i0)(0\sigma)(j0)}
   =\frac{1}{m_2}\big(\overline{\mathrm{Ric}}_{ij}-\,^{M_1}\!\mathrm{Ric}_{ij}\big)\bar g_{\gamma\sigma}.
\end{equation}
Substituting (\ref{10-6-3}) into (\ref{10-5-12}) yields
   \begin{equation} \notag
    \frac{\partial }{\partial t} \overline{\mathrm{Ric}}_{ij}
     = \bar\Delta \overline{\mathrm{Ric}}_{ij}
       +\frac{2}{m_2}\bar g^{\delta\tau}\overline{\mathrm{Ric}}_{\delta \tau}
           \big(\overline{\mathrm{Ric}}_{ij}-\,^{M_1}\!\mathrm{Ric}_{ij}\big)
    -2\bar g^{kl}\overline{\mathrm{Ric}}_{ki}\overline{\mathrm{Ric}}_{lj},
    \end{equation}
 which is (\ref{10-5-08}).

Now we calculate the evolution of
$\overline{\mathrm{Ric}}_{\alpha\beta}$. From (\ref{10-5-10}) we get
    \begin{equation} \label{10-6-4}
    \frac{\partial }{\partial t} \overline{\mathrm{Ric}}_{\alpha\beta}
     = \bar\Delta \overline{\mathrm{Ric}}_{\alpha\beta}
       +2\bar g^{\bar k\bar l\bar \gamma\bar \delta} \bar g^{\bar p\bar q \bar\sigma \bar\tau}
       \overline{\mathrm{Rm}}_{(\bar k \bar \gamma)(0 \alpha)(\bar p \bar\sigma)(0 \beta)}
       \overline{\mathrm{Ric}}_{\bar l \bar q \bar\delta \bar\tau}
 -2\bar g^{\bar k\bar l\bar \gamma \bar \delta}\overline{\mathrm{Ric}}_{\bar k 0 \bar\gamma \alpha }
 \overline{\mathrm{Ric}}_{\bar l 0 \bar\delta \beta }.
      \end{equation}
Once again using that $\bar g^{i\alpha}=0$ and
$\overline{\mathrm{Ric}}_{\alpha i}=0$, and remembering the only
non-zero terms $\overline{\mathrm{Rm}}_{(k 0)(0 \alpha)( p 0)(0
\beta)}$ and $\overline{\mathrm{Rm}}_{(0 \gamma)(0 \alpha)(0
\sigma)(0 \beta)}$, (\ref{10-6-4}) becomes
         \begin{equation} \label{10-6-5}
         \begin{split}
    \frac{\partial }{\partial t} \overline{\mathrm{Ric}}_{\alpha\beta}
     = &\bar\Delta \overline{\mathrm{Ric}}_{\alpha\beta}
      +2\bar g^{k l} \bar g^{p q }\overline{\mathrm{Rm}}_{(k 0)(0 \alpha)( p 0)(0 \beta)}
       \overline{\mathrm{Ric}}_{l q }\\
       &+2\bar g^{\gamma \delta} \bar g^{\sigma \tau}
       \overline{\mathrm{Rm}}_{(0 \gamma)(0 \alpha)(0 \sigma)(0 \beta)}
       \overline{\mathrm{Ric}}_{\delta \tau}
 -2\bar g^{ \gamma \delta}\overline{\mathrm{Ric}}_{\gamma \alpha }
 \overline{\mathrm{Ric}}_{\delta \beta }.
      \end{split}
      \end{equation}
Since (\ref{10-6-3}) gives
         \begin{equation} \label{10-6-6}
         \overline{\mathrm{Rm}}_{(k 0)(0 \alpha)( p 0)(0 \beta)}
    =\frac{1}{m_2}\big(\overline{\mathrm{Ric}}_{kp}-\,^{M_1}\!\mathrm{Ric}_{kp}\big)\bar
    g_{\alpha\beta}
  \end{equation}
and (\ref{6-30-6}) gives
       \begin{equation} \label{10-6-7}
         \overline{\mathrm{Rm}}_{(0 \gamma)(0 \alpha)(0 \sigma)(0 \beta)}
    =\lambda^2 \,^{M_2}\!\mathrm{Rm}_{\alpha\gamma\beta\sigma}
    +\lambda^2 |\mathrm{grad} \lambda|^2 \left( (g_2)_{\alpha\sigma}
    (g_2)_{\beta\gamma}-(g_2)_{\alpha\beta} (g_2)_{\gamma\sigma}
    \right),
     \end{equation}
substituting (\ref{10-6-6}) and (\ref{10-6-7}) into (\ref{10-6-5})
gives
     \begin{equation} \label{10-6-8}
         \begin{split}
    &\frac{\partial }{\partial t} \overline{\mathrm{Ric}}_{\alpha\beta}
     = \bar\Delta \overline{\mathrm{Ric}}_{\alpha\beta}
      +\frac{2}{m_2}\bar g^{k l} \bar g^{p q } \overline{\mathrm{Ric}}_{l q }
      \big(\overline{\mathrm{Ric}}_{kp}-\,^{M_1}\!\mathrm{Ric}_{kp}\big)\bar g_{\alpha\beta}
      -2\bar g^{ \gamma \delta}\overline{\mathrm{Ric}}_{\gamma \alpha}
         \overline{\mathrm{Ric}}_{\delta \beta }\\
       &\quad +2\lambda^2\bar g^{\gamma \delta} \bar g^{\sigma \tau}
        \overline{\mathrm{Ric}}_{\delta \tau}\left(\,^{M_2}\!\mathrm{Rm}_{\alpha\gamma\beta\sigma}
    +|\mathrm{grad} \lambda|^2 \big(( (g_2)_{\alpha\sigma}(g_2)_{\beta\gamma}
      -(g_2)_{\alpha\beta}(g_2)_{\gamma\sigma}\big)  \right),\\
      \end{split}
      \end{equation}
which is (\ref{10-5-09}).
\end{proof}

 To further simplify the evolution (\ref{10-5-09}) and consider perhaps
significant implication for physics, like previous subsection we
assume that $g_2$ is a fixed Einstein metric with constant $c$. We
first give a lemma.

\begin{lemma}\label{10-5} Let
      $$\bar g(x,y,t)=g_1(x,t)\oplus \lambda^2(x,t)g_2(y,t)$$
 be a smooth warped product
metric on the manifold $M_1\times M_2$, where $g_2$ is an Einstein
metric with $\,^{M_2}\!\mathrm{Ric}=c g_2$.  Then
    \begin{equation} \label{10-5-1}
    \overline{\mathrm{Ric}}_{\alpha\beta}=f(x,t)\bar   g_{\alpha\beta},
     \end{equation}
 where
   \begin{equation} \label{10-5-001}
   \begin{split}
   f&=\frac{1}{m_2}\bar g^{\alpha\beta}\overline{\mathrm{Ric}}_{\alpha\beta}\\
    & =\frac{1}{2\lambda^2}\left((4-2m_2)|\mathrm{grad}\lambda|^2-\Delta_{M_1}\lambda^2 +2c
    \right).
   \end{split}
   \end{equation}
\end{lemma}

\begin{proof} By (\ref{6-30-6}) and $\,^{M_2}\!\mathrm{Ric}=c g_2$ , we get
     \begin{equation} \notag
     \begin{split}
     \overline{\mathrm{Ric}}_{\alpha\beta}& =c(g_2)_{\alpha\beta}
 -\left( \lambda \Delta_{M_1}\lambda +(m_2-1)|\mathrm{grad}
 \lambda|^2\right)(g_2)_{\alpha\beta}\\
 &=\frac{1}{\lambda^2}\left( c-\lambda \Delta_{M_1}\lambda -(m_2-1)|\mathrm{grad}
 \lambda|^2\right)\bar g_{\alpha\beta}.
     \end{split}
\end{equation}
Note that
     $$ \lambda \Delta_{M_1}\lambda=\frac{1}{2}\Delta_{M_1}\lambda^2-|\mathrm{grad}\lambda|^2.$$
Putting these with (\ref{10-5-1}), we obtain
     $$ f=\frac{1}{2\lambda^2}\left(-\Delta_{M_1}\lambda^2
                 -2 (m_2-2)|\mathrm{grad}\lambda|^2+2c \right), $$
which is the second ``=" in (\ref{10-5-001}).

As to the first ``=" in (\ref{10-5-001}), note that
$\sum\limits_{\alpha,\beta=1}^{m_2}\bar  g_{\alpha\beta}\bar
g^{\alpha\beta}=m_2$, it quickly follows from (\ref{10-5-1}).
\end{proof}

Applying this Lemma, we can simplify (\ref{10-5-09}) to a better
expression.

\begin{theorem} \label{10-5-009}Assume that $g_2$ is a fixed
Einstein metric with $\,^{M_2}\!\mathrm{Ric}=c g_2$. Then under the
Ricci flow (\ref{9-28}), the Ricci curvature evolution equation
(\ref{10-5-09}) has another form :
  \begin{equation} \label{10-5-09'}
  \begin{split}
   \frac{\partial }{\partial t} \overline{\mathrm{Ric}}_{\alpha\beta}
     &= \bar\Delta \overline{\mathrm{Ric}}_{\alpha\beta} -\frac{2}{m_2} \bar g^{kp}
    \big(\overline{\mathrm{Ric}}_{kp}-\,^{M_1}\!\mathrm{Ric}_{kp}\big)
           \overline{\mathrm{Ric}}_{\alpha\beta} \\
     &\quad  +\frac{2}{m_2}\bar g^{k l} \bar g^{p q }
        \big(\overline{\mathrm{Ric}}_{kp}-\,^{M_1}\!\mathrm{Ric}_{kp}\big)
           \overline{\mathrm{Ric}}_{l q }\bar g_{\alpha\beta}.
   \end{split}
  \end{equation}
\end{theorem}

\begin{proof} By (\ref{10-5-1}) and (\ref{10-6-7}), the last two terms on the left-hand
side of (\ref{10-5-09}) can be transformed to
    \begin{equation} \label{10-6-10}
    \begin{split}
   (I)& :=-2\bar g^{ \gamma \delta}\overline{\mathrm{Ric}}_{\gamma \alpha}
         \overline{\mathrm{Ric}}_{\delta \beta }\\
     &\quad +2\lambda^2\bar g^{\gamma \delta} \bar g^{\sigma \tau}
        \overline{\mathrm{Ric}}_{\delta \tau}\left(\,^{M_2}\!\mathrm{Rm}_{\alpha\gamma\beta\sigma}
    +|\mathrm{grad} \lambda|^2 \big(( (g_2)_{\alpha\sigma}(g_2)_{\beta\gamma}
      -(g_2)_{\alpha\beta}(g_2)_{\gamma\sigma}\big)  \right)\\
    & =-2\bar g^{ \gamma \delta}f\bar g_{\gamma \alpha}\overline{\mathrm{Ric}}_{\delta \beta}
    +2\bar g^{\gamma \delta} \bar g^{\sigma \tau}f\bar g_{\delta\tau}
     \overline{\mathrm{Rm}}_{(0 \gamma)(0 \alpha)(0 \sigma)(0 \beta)}\\
    &=2f\left(-\overline{\mathrm{Ric}}_{\alpha \beta}
    +\bar g^{\gamma\sigma }\overline{\mathrm{Rm}}_{(0 \gamma)(0 \alpha)(0 \sigma)(0 \beta)}\right).
      \end{split}
  \end{equation}
By the definition of the Ricci curvature and the only non-zero terms
$\overline{\mathrm{Rm}}_{(k 0)(0 \alpha)( p 0)(0 \beta)}$ and \\
$\overline{\mathrm{Rm}}_{(0 \gamma)(0 \alpha)(0 \sigma)(0 \beta)}$,
we have
    \begin{equation} \label{10-6-11}
    \overline{\mathrm{Ric}}_{\alpha\beta}:= \bar g^{\bar k\bar p \bar\gamma \bar\sigma}
          \overline{\mathrm{Rm}}_{(\bar k \bar \gamma)(0 \alpha)(\bar p \bar\sigma)(0 \beta)}
       =\bar g^{kp}\overline{\mathrm{Rm}}_{(k 0)(0 \alpha)( p 0)(0 \beta)}
       +\bar g^{\gamma \sigma} \overline{\mathrm{Rm}}_{(0 \gamma)(0 \alpha)(0 \sigma)(0
       \beta)}.
    \end{equation}
Combining (\ref{10-6-11}) and (\ref{10-6-6}), we get
     \begin{equation} \label{10-6-12}
    \bar g^{\gamma \sigma }
       \overline{\mathrm{Rm}}_{(0 \gamma)(0 \alpha)(0 \sigma)(0 \beta)}
    =\overline{\mathrm{Ric}}_{\alpha\beta} -\frac{1}{m_2}\bar g^{kp}\bar g_{\alpha\beta}
    \big(\overline{\mathrm{Ric}}_{kp}-\,^{M_1}\!\mathrm{Ric}_{kp}\big).
\end{equation}
Substituting (\ref{10-6-12}) into (\ref{10-6-10}) and using the
relation (\ref{10-5-1}), we obtain
     \begin{equation} \label{10-6-13}
         \begin{split}
       (I)&=-\frac{2}{m_2}(f \bar g_{\alpha\beta})\bar g^{kp}
    \big(\overline{\mathrm{Ric}}_{kp}-\,^{M_1}\!\mathrm{Ric}_{kp}\big)\\
    & =-\frac{2}{m_2} \bar g^{kp} \overline{\mathrm{Ric}}_{\alpha\beta}
    \big(\overline{\mathrm{Ric}}_{kp}-\,^{M_1}\!\mathrm{Ric}_{kp}\big).\\
         \end{split}
      \end{equation}
Finally, substituting (\ref{10-6-13}) into (\ref{10-5-09}) gives
       \begin{equation} \notag
         \begin{split}
    \frac{\partial }{\partial t} \overline{\mathrm{Ric}}_{\alpha\beta}
     &= \bar\Delta \overline{\mathrm{Ric}}_{\alpha\beta}
      +\frac{2}{m_2}\bar g^{k l} \bar g^{p q } \overline{\mathrm{Ric}}_{l q }
      \big(\overline{\mathrm{Ric}}_{kp}-\,^{M_1}\!\mathrm{Ric}_{kp}\big)\bar g_{\alpha\beta}\\
      & \quad -\frac{2}{m_2} \bar g^{kp} \overline{\mathrm{Ric}}_{\alpha\beta}
    \big(\overline{\mathrm{Ric}}_{kp}-\,^{M_1}\!\mathrm{Ric}_{kp}\big),
      \end{split}
      \end{equation}
which is (\ref{10-5-09'}).
\end{proof}

 Now we return to find out the interesting evolution equation
of $f(x,t)$.

\begin{theorem} \label{10-6-13'} Assume that $g_2$ is a fixed
Einstein metric with $\,^{M_2}\!\mathrm{Ric}=c g_2$. Then under the
Ricci flow (\ref{9-28}), $f(x,t)$ satisfies the evolution equation
  \begin{equation} \label{10-6-14}
   \begin{split}
  \frac{\partial}{\partial t} f &=\bar\Delta\ f+2f^2-\frac{2}{m_2} \bar g^{kp}
    \big(\overline{\mathrm{Ric}}_{kp}-\,^{M_1}\!\mathrm{Ric}_{kp}\big) f  \\
   &\hskip 1cm  +\frac{2}{m_2} \bar g^{k l} \bar g^{p q }
        \big(\overline{\mathrm{Ric}}_{kp}-\,^{M_1}\!\mathrm{Ric}_{kp}\big)
           \overline{\mathrm{Ric}}_{l q }.
      \end{split}
  \end{equation}
\end{theorem}

\begin{proof}
  Since $\bar g^{i\alpha}=0$ and
  $\overline{\mathrm{Ric}}_{\alpha i}=0$, by (\ref{10-5-001}) we have
      \begin{equation} \label{10-6-15}
  \begin{split}
  \frac{\partial}{\partial t} f
   &=\frac{\partial}{\partial t}\big( \frac{1}{m_2}\bar g^{\alpha\beta}
     \overline{\mathrm{Ric}}_{\alpha\beta} \big)\\
  &=\frac{1}{m_2}\left(\frac{\partial}{\partial t}\bar g^{\alpha\beta}\right)
    \overline{\mathrm{Ric}}_{\alpha\beta}+\frac{1}{m_2}\bar g^{\alpha\beta}
      \left(\frac{\partial}{\partial t}\overline{\mathrm{Ric}}_{\alpha\beta}\right).
      \end{split}
  \end{equation}
Note that
    \begin{equation} \notag
  \begin{split}
  0=\frac{\partial}{\partial t}(\delta_{\alpha\gamma})& =\frac{\partial}{\partial
      t}\left( \bar g^{\alpha\beta} \bar g_{\beta\gamma} \right)\\
      &=\frac{\partial}{\partial t}\left( \bar g^{\alpha\beta} \right)\bar g_{\beta\gamma}
       +\bar g^{\alpha\beta}\frac{\partial}{\partial t}\left(  \bar g_{\beta\gamma} \right).
      \end{split}
  \end{equation}
Combining this and (\ref{9-28}), yields
    \[ \frac{\partial}{\partial t} \bar g^{\alpha\beta}
      =2\bar g^{\alpha\tau} \bar g^{\beta \sigma}\overline{\mathrm{Ric}}_{\tau\sigma}. \]
 Thus substituting this and (\ref{10-5-09'}) into (\ref{10-6-15}) gives
       \begin{equation} \label{10-6-16}
  \begin{split}
  \frac{\partial}{\partial t} f
 & =\frac{2}{m_2}\bar g^{\alpha\tau} \bar g^{\beta \sigma}\overline{\mathrm{Ric}}_{\tau\sigma}
      \overline{\mathrm{Ric}}_{\alpha\beta}\\
     & \quad  +\frac{1}{m_2} \bar g^{\alpha\beta}
     \bigg(   \bar\Delta \overline{\mathrm{Ric}}_{\alpha\beta} -\frac{2}{m_2} \bar g^{kp}
    \big(\overline{\mathrm{Ric}}_{kp}-\,^{M_1}\!\mathrm{Ric}_{kp}\big)
           \overline{\mathrm{Ric}}_{\alpha\beta} \\
     &\hskip 2cm +\frac{2}{m_2}\bar g^{k l} \bar g^{p q }
        \big(\overline{\mathrm{Ric}}_{kp}-\,^{M_1}\!\mathrm{Ric}_{kp}\big)
           \overline{\mathrm{Ric}}_{l q }\bar g_{\alpha\beta}\bigg)\\
      &=\frac{2}{m_2}\bar g^{\alpha\tau} \bar g^{\beta \sigma}(f \bar
      g_{\tau\sigma})(f \bar g_{\alpha\beta})\\
      & \quad  +\frac{1}{m_2}\bar\Delta
       \left(\bar g^{\alpha\beta}(f \bar g_{\alpha\beta})\right)
      -\frac{2}{m_2^2} \bar g^{kp}
    \big(\overline{\mathrm{Ric}}_{kp}-\,^{M_1}\!\mathrm{Ric}_{kp}\big)
          \bar g^{\alpha\beta} (f \bar g_{\alpha\beta}) \\
    &\hskip 1cm +\frac{2}{m_2^2} \bar g^{k l} \bar g^{p q }
        \big(\overline{\mathrm{Ric}}_{kp}-\,^{M_1}\!\mathrm{Ric}_{kp}\big)
           \overline{\mathrm{Ric}}_{l q }(\bar g^{\alpha\beta}\bar  g_{\alpha\beta}).\\
   &=\frac{2}{m_2}f^2 m_2+\frac{1}{m_2}\bar\Delta( m_2f)
      -\frac{2}{m_2^2} \bar g^{kp}
    \big(\overline{\mathrm{Ric}}_{kp}-\,^{M_1}\!\mathrm{Ric}_{kp}\big)(fm_2)\\
  &\hskip 1cm  +\frac{2}{m_2^2} \bar g^{k l} \bar g^{p q }
        \big(\overline{\mathrm{Ric}}_{kp}-\,^{M_1}\!\mathrm{Ric}_{kp}\big)
           \overline{\mathrm{Ric}}_{l q }\cdot m_2,
      \end{split}
  \end{equation}
that is
   \begin{equation} \notag
  \begin{split}
  \frac{\partial}{\partial t} f &=2f^2+ \bar\Delta\ f-\frac{2}{m_2} \bar g^{kp}
    \big(\overline{\mathrm{Ric}}_{kp}-\,^{M_1}\!\mathrm{Ric}_{kp}\big) f  \\
   &\hskip 1cm  +\frac{2}{m_2} \bar g^{k l} \bar g^{p q }
        \big(\overline{\mathrm{Ric}}_{kp}-\,^{M_1}\!\mathrm{Ric}_{kp}\big)
           \overline{\mathrm{Ric}}_{l q }.
           \end{split}
   \end{equation}
This is the desired equality (\ref{10-6-14}).
\end{proof}

\begin{remark} In this theorem, if we take $M_1=\mathbb R$, then
$\,^{M_1}\!\mathrm{Ric}=0$. This time, since $m_1=1$, by changing
the indices $i, j, k,l, p, q$ into a same notation ``$x$", then
(\ref{10-5-08}), (\ref{10-5-09'}) and (\ref{10-6-14}) are
respectively rewritten as
  \begin{equation} \notag
  \begin{split}
  &  \frac{\partial }{\partial t} \overline{\mathrm{Ric}}_{xx}
     = \bar\Delta \overline{\mathrm{Ric}}_{xx}
       +\frac{2}{m_2}\bar g^{\alpha\beta}\overline{\mathrm{Ric}}_{\alpha\beta}
       \overline{\mathrm{Ric}}_{xx} -2\bar g^{xx}(\overline{\mathrm{Ric}}_{xx})^2,\\
    &  \frac{\partial }{\partial t} \overline{\mathrm{Ric}}_{\alpha\beta}
     = \bar\Delta \overline{\mathrm{Ric}}_{\alpha\beta} -\frac{2}{m_2} \bar g^{xx}
    \overline{\mathrm{Ric}}_{xx} \overline{\mathrm{Ric}}_{\alpha\beta}
     +\frac{2}{m_2}\bar g^{xx} \bar g^{xx } (\overline{\mathrm{Ric}}_{xx})^2 \bar g_{\alpha\beta}, \\
    &\frac{\partial}{\partial t} f =2f^2+ \bar\Delta\ f-\frac{2}{m_2} \bar g^{xx}
    \overline{\mathrm{Ric}}_{xx} f  +\frac{2}{m_2} \bar g^{xx} \bar g^{xx }
       (\overline{\mathrm{Ric}}_{xx})^2,
           \end{split}
   \end{equation}
which are exactly (4.3)-(4.4) in Proposition 4.1 in \cite{Si}.
\end{remark}

\begin{remark} Under HGF, the evolution equations
for Ricci curvature on a single manifold are much more complicated
when compared with the case under RF, because they involve some
complex terms such as  $B
\big(X,B(X,Y)\big):=\frac{\partial}{\partial t}\big(\nabla_X
(\frac{\partial}{\partial t}\nabla_Y Z)\big)$ and the unknown term $
\frac{\partial g(t)}{\partial t}$( see Theorem 1.4 or Theorem 1.1 in
\cite{Lu}, or Theorem 5.2 in \cite{DKL} ), let alone on the warped
product manifold. Therefore, it is very hard to gain some novel
evolution equations for Ricci curvature on warped product manifold
$\bar M$ when we still want to follow the introduced approach under
the RF. Taking into account the just mentioned reason and our
present technique, in this paper we put this issue aside for a
moment.
\end{remark}

\end{document}